\documentclass[final]{siamltex}
\usepackage{amsmath, amssymb, amsfonts}
\usepackage{graphicx,color,caption,subcaption}
\usepackage{diagbox}
\usepackage[ruled]{algorithm2e}
\usepackage{algorithmic}
\usepackage{enumitem}
\usepackage{multirow}
\usepackage{epstopdf}
\usepackage{placeins}

\newcommand{\Tr}{\mathrm{Tr}}

\newcommand{\norm}[1]{\lVert#1\rVert}

\newcommand{\bra}[1]{\langle#1\vert}
\newcommand{\ket}[1]{\vert#1\rangle}

\newcommand{\ie}{\textit{i.e.}{~}}
\newcommand{\eg}{\textit{e.g.}{~}}
\newcommand{\Or}{\mathcal{O}}

\newcommand{\RR}{\mathbb{R}}

\newcommand{\re}{r_{\ve}}
\newcommand{\dre}{\dot{r}_{\ve}}
\newcommand{\ddre}{\ddot{r}_{\ve}}

\newcommand{\xe}{x_{\ve}}
\newcommand{\dxe}{\dot{x}_{\ve}}
\newcommand{\ddxe}{\ddot{x}_{\ve}}

\newcommand{\ret}{r_{\xi}}
\newcommand{\dret}{\dot{r}_{\xi}}
\newcommand{\ddret}{\ddot{r}_{\xi}}

\newcommand{\pet}{p_{\xi}}
\newcommand{\dpet}{\dot{p}_{\xi}}

\newcommand{\xet}{x_{\xi}}
\newcommand{\dxet}{\dot{x}_{\xi}}
\newcommand{\ddxet}{\ddot{x}_{\xi}}

\newcommand{\yet}{y_{\xi}}
\newcommand{\dyet}{\dot{y}_{\xi}}

\newcommand{\rave}{\overline{r}}
\newcommand{\drave}{\dot{\overline{r}}}

\newcommand{\pave}{\overline{p}}
\newcommand{\dpave}{\dot{\overline{p}}}

\numberwithin{equation}{section}
\numberwithin{figure}{section}
\newtheorem{thm}{\protect\theoremname}

\newtheorem{prop}[thm]{\protect\propositionname}

\providecommand{\corollaryname}{Corollary}
\providecommand{\lemmaname}{Lemma}
\providecommand{\propositionname}{Proposition}
\providecommand{\remarkname}{Remark}
\providecommand{\theoremname}{Theorem}

\title{Convergence of Stochastic-extended Lagrangian molecular
dynamics method for polarizable force field simulation}

\author{Dong An
\thanks{Department of Mathematics, University of California, Berkeley,
CA 94720. Email: \texttt{dong\_an@berkeley.edu}}
\and
Sara Y. Cheng
\thanks{Department of Chemistry, University of California, Berkeley,
CA 94720. 
Email: \texttt{sycheng@berkeley.edu}}
\and
Teresa Head-Gordon
\thanks{Department of Chemistry, Bioengineering, and Chemical and Biomolecular Engineering, University of California, Berkeley,
and Chemical Sciences Division, Lawrence Berkeley National Laboratory,
CA 94720. 
Email: \texttt{thg@berkeley.edu}}
\and
Lin Lin
\thanks{Department of Mathematics, University of California, Berkeley, and Computational Research Division, Lawrence Berkeley National Laboratory, Berkeley, CA 94720. Email: \texttt{linlin@math.berkeley.edu}}
\and
Jianfeng Lu
\thanks{
Department of Mathematics, Department of Physics, and Department of
Chemistry,
Duke University, Durham, NC 27708, USA. Email: {\tt jianfeng@math.duke.edu}
}
}

\begin{document}

\global\long\def\ve{\varepsilon}
\global\long\def\R{\mathbb{R}}
\global\long\def\Rn{\mathbb{R}^{n}}
\global\long\def\Rd{\mathbb{R}^{d}}
\global\long\def\E{\mathbb{E}}
\global\long\def\P{\mathbb{P}}
\global\long\def\bx{\mathbf{x}}
\global\long\def\vp{\varphi}
\global\long\def\ra{\rightarrow}
\global\long\def\smooth{C^{\infty}}
\global\long\def\Tr{\mathrm{Tr}}
\global\long\def\bra#1{\left\langle #1\right|}
\global\long\def\ket#1{\left|#1\right\rangle }

\maketitle

\begin{abstract}
%Polarizable force field plays an important role in classical molecular
%dynamics simulation, but the computational cost is often much larger
%than that of rigid force fields due to the need of solving a large
%linear system at each time step. 
Extended Lagrangian molecular dynamics (XLMD) is a general method for performing molecular dynamics simulations using quantum and classical many-body potentials. Recently several new XLMD schemes have been proposed and tested on several classes of many-body polarization models such as induced dipoles or Drude charges, by creating an auxiliary set of these same degrees of freedom that are
reversibly integrated through time. This gives rise to a
singularly perturbed Hamiltonian system that provides a good
approximation to the time evolution of the real mutual polarization field. To
further improve upon the accuracy of the XLMD dynamics, and to
potentially extend it to other many-body potentials,
%\LL{I do not
%understand why adding a Langevin thermostat helps the generalization to
%other potentials} %THG{you can remove this if you wish, but I put it in
%because extensions to AIMD or charge equilibration are more difficult,
%and maybe Langevin will be less sensitive to initial conditions as
%stated below; and I said "potentially";)}
we introduce a stochastic modification which leads to a set of singularly perturbed Langevin
equations with degenerate noise. We prove that the resulting Stochastic-XLMD
converges to the accurate dynamics, and the convergence rate is
both optimal and is independent of the accuracy of the initial polarization field. We carefully study the scaling of the damping factor and
numerical noise for efficient numerical simulation for Stochastic-XLMD,
and we demonstrate the effectiveness of the method for model polarizable
force field systems.
\end{abstract}


\section{Introduction}

Molecular dynamics (MD) simulations often require solving a linear or nonlinear system repeatedly for certain latent variables.  For % quantum molecular dynamics simulation (or
\textit{ab initio} molecular dynamics simulations~\cite{Martin2004}, the latent variable is the electron density. At each MD step, the electron density needs to be obtained by the self-consistent solution of the Kohn-Sham
equations~\cite{HohenbergKohn1964,KohnSham1965}, which
are a set of nonlinear eigenvalue equations. In classical
molecular dynamics simulation with a polarizable force field~\cite{Demerdash2014,Albaugh2016}, it is the induced dipole or Drude charge that needs to be evaluated through the solution of
a linear system, typically solved to self-consistency for large systems. 

In a simplified mathematical setting, the problem can be stated as
follows. Let $r\in \R^{d}$ be the collection of atomic positions, and
$x\in \R^{d'}$ be the collection of latent variables such as the induced
dipoles. Let $U(r)$ be a smooth external potential
field involving only the atomic positions, which gives the external force
\[F(r) = - \frac{\partial U}{\partial r}(r){.} \]
Let $Q(r,x)$ be the
interaction energy involving both the atomic position and the latent
variable, and we assume $Q$ is smooth. 
For a given $r$, the latent variable $x$ is determined by the following equation
\begin{equation}\label{eqn:general_algebraic}
 \frac{\partial Q}{\partial x}(r,x)=0.
\end{equation}
We assume the solution to \eqref{eqn:general_algebraic} is unique for all $r\in\R^d$.  
The molecular dynamics simulation requires the solution of the following differential-algebraic equations (DAE) system 
\begin{subequations}\label{eqn:General_Dynamics}
\begin{align}
    \ddot{r}_{\star}(t) &= F(r_{\star}(t)) - \frac{\partial Q}{\partial
    r}(r_{\star}(t),x_{\star}(t)), \label{eqn:General_Dynamics_a}\\
    0 &= - \frac{\partial Q}{\partial x}(r_{\star}(t),x_{\star}(t)), \label{eqn:General_Dynamics_b}
\end{align}
\end{subequations}
subject to certain initial conditions
$r_{\star}(0),\dot{r}_{\star}(0)$.  Here the subscript $\star$ is used to
indicate the exact solution of Eq.~\eqref{eqn:General_Dynamics}.  Note
that the initial condition for $x$ is not needed since it can be
determined from $r_{\star}(0)$ through
Eq.~\eqref{eqn:General_Dynamics_b} (recall that a unique solution is
assumed). To simplify the notation, we assume the mass is unity for
all atomic degrees of freedom. Unless otherwise specified, we shall
drop the explicit dependence on the time variable $t$ below, and
without loss of generality we assume $d'=d$.

The polarizable force field simulation in classical molecular dynamics
is an interesting and a particularly suitable case for analysis, since
$Q(r,x)$ becomes just a quadratic function with respect to the polarization field $x$: 
\begin{equation}\label{eqn:Hamiltonian_Quadratic}
    Q(r,x) = \frac{1}{2}x^{\top} A(r)x - b(r)^{\top}x.
\end{equation}
Here for each $r$, $A(r)$ is a positive definite matrix, with its smallest 
eigenvalue uniformly  bounded above $0$. Hence the solution $x(r)$ is unique for all $r$.
We also assume the mappings $b:\R^{d}\ra\R^{d}$ and
$A:\R^{d}\ra\mathcal{S}_{++}^{d}$ are smooth.
Eq.~\eqref{eqn:general_algebraic} is then reduced to a simple linear
equation
\begin{equation}
  A(r) x = b(r).
  \label{eqn:dipole_equation}
\end{equation}
%where  
%The potential energy depending on the polarization takes the form 
%\[
%-\frac12 x^{\top} b(r) = -\frac12 b^{\top}(r) A^{-1}(r) b(r),
%\]
%which can be obtained from the function value of the 
%minimization problem $\inf_{x\in\R^{d}} Q(r,x)$, with $Q$ being quadratic with respect to $x$ as 
This will greatly simplify our analysis in the results below. 

Eq.~\eqref{eqn:General_Dynamics_b} or~\eqref{eqn:dipole_equation} is an algebraic system that needs to be solved at each MD time step. In molecular dynamics simulation, we are generally more interested in the accuracy of the trajectory of atoms $r(t)$ than that of the latent variables $x(t)$. In the past decade, new types of 
integrators called the extended Lagrangian Born-Oppenheimer molecular dynamics (XL-BOMD) method~\cite{Niklasson2008}
(initially called the time reversible molecular dynamics (TRMD)
method~\cite{Niklasson2006}) have been developed. The main idea of
XL-BOMD is to write down an extended Lagrangian for the latent variable. Instead of being solved through an algebraic system at each time
step, the latent variables are \textit{evolved} together with the atomic positions. XL-BOMD
differs from previous extended Lagrangian molecular dynamics (XLMD)
integration schemes such as Car-Parrinello molecular
dynamics~\cite{CarParrinello1985} by eliminating the coupling or mass
parameter of the latent variables. Numerical results demonstrate that
this strategy can significantly reduce the number of self-consistent
iterations~\cite{Niklasson2006,Niklasson2008,AlbaughDemerdashHead-Gordon2015},
and in some cases fully eliminate the need for performing
self-consistent iteration altogether~\cite{Niklasson2012,AlbaughNiklassonHead-Gordon2017,AlbaughHead-Gordon2017}.

Following Eq.~\eqref{eqn:General_Dynamics}, the extended Lagrangian for
the XL-BOMD approach takes the general form
\begin{equation}
  L_{\ve} = \frac12 |\dre|^2 + \frac{\varepsilon }{2} |\dxe|^2 - U(\re) -
  Q(\re,\xe).
  \label{eqn:extended_lagrangian}
\end{equation}
%and in the limit that the latent variable mass $m \to 0$  yields \JL{I don't see $m$  above?}
 The corresponding Euler-Lagrange
equation yields
\begin{equation}\label{eqn:General_Deterministic}
  \begin{split}
    \ddre &= F(\re) - \frac{\partial Q}{\partial r}(\re,\xe), \\
    \varepsilon \ddxe &= -\frac{\partial Q}{\partial x}(\re,\xe).
  \end{split}
\end{equation}
%It is important to digress for a moment to explain that $\ve$ can be
%related the square of the discretized time step \LL{I think it should be
%the square of the discretized time step, instead of the square of the
%inverse?}%THG: yes - sorry - we usually isolate it on the RHS of the equations 
%Typically the term `Born-Oppenheimer' is only used in the
%context of quantum simulation. With some abuse of terminology, here we
%refer to Eq.~\eqref{eqn:General_Deterministic} as the XL-BOMD dynamics
%in a more general context.
Note that initial conditions for $\xe$ and $\dxe$ are needed for
\eqref{eqn:General_Deterministic}: $\xe(0)$ is often prescribed by
solving the algebraic equation
$0 = - \partial Q/\partial x(\re(0), \xe(0))$ and $\dxe(0)$ can be obtained 
by differentiating Eq.~\eqref{eqn:General_Dynamics_b} and then let $t = 0$. %\JL{please check the previous sentence I added} 
%\DA{maybe the usual choice for $\dxe(0)$ is also ``exact'': taking time derivative in (1.2b) and let $t = 0$ can give an equation of $\dxe(0)$, which can be solved with some iterations. }
Eq.~\eqref{eqn:General_Deterministic} is a Hamiltonian system, and it
can be discretized with symplectic or time-reversible integrators to
obtain long time stability~\cite{HairerLubichWanner2006}.  When $\ve$
is sufficiently small, we may expect that the solution of
Eq.~\eqref{eqn:General_Deterministic} to closely follow the exact
dynamics.  On the other hand, the value of $\sqrt{\ve}$ (which may
include an additional multiplicative factor that can be viewed as a mixing parameter) provides an upper
bound of the time step of the numerical
integrator~\cite{Niklasson2008,AlbaughNiklassonHead-Gordon2017,AlbaughHead-Gordon2017}.
Therefore it is desirable to choose $\ve$ not too small in practice.
Although Eq.~\eqref{eqn:General_Deterministic} introduces a systematic
error in terms of $\ve$ per step, hence sacrificing the accuracy of
$x(t)$ to some degree, with a properly chosen $\ve$, XL-BOMD often
outperforms the discretized original dynamics in terms of efficiency
and long time stability while still maintaining the accuracy for
$r(t)$.

From a mathematical point of view, the equations of motion
\eqref{eqn:General_Deterministic} can be viewed as a set of singularly
perturbed equations. To the best of our knowledge, the convergence of
the general XL-BOMD schemes~\eqref{eqn:General_Deterministic} as
$\ve\to 0$ has not been established other than in the linear response
regime~\cite{LinLuShao2014}, where the coupled system can be exactly
diagonalized. It is difficult to generalize the analysis to nonlinear systems. Another
issue associated with Eq.~\eqref{eqn:General_Deterministic} is that
the equation is free of dissipation. Hence numerical error introduced
by the initial condition for $\xe$ as well as external perturbation
during the simulation will be memorized throughout the simulation. To
overcome this problem, a number of approaches have been
developed. Niklasson and co-workers have added well-designed
dissipation terms to the dynamics, and though often effective, they
necessarily break time reversibility~\cite{Niklasson2009}. Albaugh et
al. have instead introduced Nose-Hoover thermostats for the latent
variables, which greatly improves the robustness of XL-BOMD since the
extended system thermostat variables can also evolve with
time-reversible
integration~\cite{AlbaughDemerdashHead-Gordon2015}. With careful
consideration of extended system thermostat formulations or
dissipation that is time-reversible to high order, the resulting
numerical schemes for XL-BOMD can be highly competitive for MD
simulations, e.g. with polarizable force fields.

In this paper we consider an alternative way to account for the needed
fluctuation and dissipation by introducing a stochastic thermostat  through the
following modified XL-BOMD scheme: %\JL{I am not sure if we shall say that Langevin thermostat is merely used to introduce dissipation; as of course it also brings fluctuation. Perhaps we shall say instead ``stabilization''?}
%\DA{not sure the meaning of ``stabilization'' here, maybe directly use ``fluctuation and dissipation'' instead of ``dissipation''? }
\begin{subequations}\label{eqn:General_Stochastic}
\begin{align}
    \ddret =& F(\ret) - \frac{\partial Q}{\partial r}(\ret,\xet),
    \label{eqn:General_Stochastic_a}\\
    \ve \ddxet =& - \frac{\partial Q}{\partial x}(\ret,\xet) 
    - \sqrt{\ve}\gamma \dxet + \sqrt{2\gamma T}\ve^{1/4} \dot{W}.
    \label{eqn:General_Stochastic_b}
\end{align}
\end{subequations}
Here the subscript $\xi = (\ve,T,\gamma)$ denotes the
set of parameters. $T>0$ is an artificial temperature for the latent variable,
$\gamma$ is an artificial friction parameter, and $\dot{W}(t)$ is the
white noise. Note that the noise is degenerate and is applied only to
the $x$ component. The scaling factors of the friction term and the
noise with respect to $\ve$ are the proper scaling relations due to
the fluctuation-dissipation relation~\cite{LeimkuhlerMatthews2015,Pavliotis2014}.
%\JL{I think we shall cite here the book by Pavliotis here instead} 
Eq.~\eqref{eqn:General_Stochastic_b} is a Langevin equation, and thus
the system will be referred to as the Stochastic-XLMD scheme in the
following discussion. %\JL{the name Stochastic-XLMD seems quite long; is S-XLMD a good choice?} 

  Compared to the Nose-Hoover thermostat, the use
of a Langevin thermostat has better ergodicity properties and hence
facilitates our analysis. The Langevin thermostat does not require
propagation of auxiliary variables used in the Nose-Hoover thermostat,
and hence is also computationally less expensive.

Substituting $Q(r,x)$ from Eq.~\eqref{eqn:Hamiltonian_Quadratic} into
Eq.~\eqref{eqn:general_algebraic},~\eqref{eqn:General_Deterministic},
and~\eqref{eqn:General_Stochastic}, we arrive at the exact dynamics,
XL-BOMD, and Stochastic-XLMD for the polarizable force field,
respectively. In particular, \eqref{eqn:General_Deterministic}
together with the form~\eqref{eqn:Hamiltonian_Quadratic} provides an
alternative derivation of the recently developed inertial extended
Lagrangian without self-consistent field iteration (iEL/0-SCF)
method~\cite{AlbaughNiklassonHead-Gordon2017,AlbaughHead-Gordon2017}.
The form of the Stochastic-XLMD for polarizable force fields will be
given explicitly in Eq.~\eqref{eqn:SDE} in
section~\ref{sec:timeaverage}.

%\JL{perhaps write down the system here just for readers' convenience? or
%point to the equations in Section 2?}
% \JL{We shall something about why we choose Langevin
%thermostat for the latent variable rather than Nose-Hoover thermostat
%approach mentioned above? (of course the former is much easier to
%analyze) Do we view this as a surrogate for the Nose-Hoover approach
%that is easier to analyze or a practically competitive strategy?}

%The minimizer is unique, and Eq.~\eqref{eqn:general_algebraic} becomes
%the following simplified set of equations
%\begin{subequations}\label{eqn:Forcefield_Dynamics}
%\begin{align}
%    \ddot{r}_{\star}(t) &= F(r_{\star}(t))
%    -\left[\frac{1}{2}x_{\star}^{\top}(t)\left[\frac{\partial A}{\partial
%    r}(r_{\star}(t))\right]x_{\star}(t)-x_{\star}^{\top}(t)\frac{\partial
%    b}{\partial r}(r_{\star}(t))\right]     {,} \\
%    0 &= b(r_{\star}(t)) - A(r_{\star}(t))x_{\star}(t).
%\end{align}
%\end{subequations}

%Even with accelerated convergence
%schemes such as conjugate gradient type methods, these equations typically require tens or more
%iterations per MD step to avoid a sizable drift of the energy. Since a
%typical MD simulation consists of thousands to millions of time steps,
%the number of iteration steps may become a significant bottleneck.

\smallskip 
\noindent\textbf{Contribution:}
The main contribution of this paper is to prove that for the
polarizable force field model, the atomic dynamics of Stochastic-XLMD
method converges to the exact dynamics as $\ve,T\to 0$. More
specifically, under proper assumptions, we prove the following
bounds for 2-norm errors: 
\begin{align*}
  &\mathbf{E}\Bigl(\sup_{0 \leq t \leq t_f} \lvert \ret(t) -
  r_{\star}(t)\rvert \vee \lvert \pet(t) - p_{\star}(t)\rvert \Bigr) \leq C \left(\ve^{1/2} + \ve^{1/4}T^{1/2} + T\right) {.}% \\
  % &\mathbf{E}\left(\sup_{0 \leq t \leq t_f}\vert \pet(t) - p_{\star}(t)\vert\right) \leq C \left(\ve^{1/2} + \ve^{1/4}T^{1/2} + T\right) {.}
\end{align*}
Here $p_{\star}(t)=\dot{r}_{\star}(t)$ is the momentum for the exact
dynamics, and similarly $\pet(t)=\dot{r}_{\xi}(t)$. $a \vee b$ stands
for the maximum of $a$ and $b$. Our proof is based on
the method of averaging (see \eg~\cite{PavliotisStuart2008}).
In particular, when the temperature $T\sim \ve^{1/2}$, the convergence
rates for both $\ret$ and $\pet$ are $\Or(\ve^{1/2})$. Since $\ve^{-1/2}$ is
proportional to the highest frequency of the latent dynamics
$\xet$, the convergence rate is optimal. 

One feature of the Stochastic-XLMD method is that in contrast to the
behavior of the XL-BOMD method, the convergence rate does not depend
on the accuracy of the initial condition of the latent variable
$x(0),\dot{x}(0)$ (from solving the algebraic equation at $t =
0$). This is because Stochastic-XLMD has a damping factor, and the
numerical error on the latent variable can only affect the dynamics
within a finite time interval.  We study the efficiency of
Stochastic-XLMD with respect to the choice of the damping factor
$\gamma$, which indicates that $\gamma$ should be generally 
$\Or(1)$ in order to minimize the numerical error. This confirms
the proper scaling relation with respect to $\ve$ in
Eq.~\eqref{eqn:General_Stochastic}, and that the dissipation term
$\gamma \dxet$ should not be too large in order to avoid a strong
perturbation of the time-reversible microcanonical
dynamics~\cite{Niklasson2009}.  Numerical results for model
polarizable force field systems verify our theoretical
estimates.  We also performed numerical results for systems with
non-quadratic interaction energy with respect to latent variable $x$,
and the numerical behavior is similar to that of the polarizable force
field models.

%For clarity our proof is
%presented for the interaction energy of the
%form~\eqref{eqn:Hamiltonian_Quadratic}, which is quadratic for the latent degrees of
%freedom. 
%But the results can be modified to accommodate more general
%Hamiltonian forms. \JL{it might be overclaiming as for general $W$, we don't have such explicit rate estimates?}

%\JL{Do we want to consider the non-quadratic case numerically?}

% LL: Moved related works after the main theorem according to Jianfeng's
% suggestion.

%\smallskip 
%\noindent\textbf{Related works:} 
%\LL{Jianfeng: can you add some discussions related to the time averaging
%  of Langevin type systems?}
%\JL{I am not sure if we shall discuss here works on averaging, in particular those on technical aspects. I feel they should be moved to Section 2 after the main theorem? Here perhaps we shall just say something as averaging of SDE has been well studied in the literature ... In fact, do we need a section on related works?}

\smallskip 
\noindent\textbf{Organization:} 
The rest of the paper is organized as follows.  We study the limit
when $\ve,T\to0$ in terms of time averaging and state the main result,
Theorem~\ref{thm:Main_Theorem} in section~\ref{sec:timeaverage}.  The
proof of the main theorem %Theorem~\ref{thm:Main_Theorem}
is given in section~\ref{sec:proof}. The results are justified by
numerical results in section~\ref{sec:numer}, followed by conclusion
and discussion in section~\ref{sec:conclusion}.

\section{Method of time averaging}\label{sec:timeaverage}

%In this paper, we restrict ourselves with the quadratic Hamiltonian
%This corresponds to the classical polarization model, 
%which was studied in xxx.  
%Quadratic Hamiltonian much simplifies the analysis because 
%it satisfies all the growth assumptions which are required to 
%assure regularity, and, more importantly, 
%the corresponding Langevin dynamics becomes a linear system of SDE 
%whose solution can be written down explicitly. 

%We will set up the model mathematically, 
%introduce the idea of time averaging and state the main theorem 
%in this section. 

In the discussion below, we denote the momentum variables by $p$ and $y$, such
that $p_{\star},\pet$ are the first order time derivatives of $r_{\star},\ret$, 
respectively, and $\yet = \sqrt{\ve}\dxet$ is 
the rescaled time derivative of $\xet$. 
%Let $b:\R^{d}\ra\R^{d}$, $F:\R^{d}\ra\R^{d}$ 
%and $A:\R^{d}\ra\mathcal{S}_{++}^{d}$ are
%$C^{\infty}$ maps, and assume that there exists $C>0$ such that $A(r)\succeq C^{-1}$ for all $r\in\R^{d}$. 
For the polarizable force field model with a quadratic interaction energy~\eqref{eqn:Hamiltonian_Quadratic},
the exact dynamics~\eqref{eqn:General_Dynamics} can be
rewritten as
\begin{equation}\label{eqn:Exact_ODE}
    \begin{split}
    \dot{r}_{\star} &= p_{\star}   {,}\\
    \dot{p}_{\star} &= F(r_{\star}) -\left[\frac{1}{2}x_{\star}^{\top}\frac{\partial A}{\partial r}(r_{\star})x_{\star}-\frac{\partial b}{\partial r}(r_{\star})^{\top}x_{\star}\right]     {,} \\
    0 &= b(r_{\star}) - A(r_{\star})x_{\star}, 
    \end{split}
\end{equation}
with initial values $r_{\star}(0)$ and $p_{\star}(0)$. 
Since the evolution of the latent variable $x_{\star}$ is determined by
the evolution of $r_{\star}$ via $x_{\star}=A(r_{\star})^{-1}b(r_{\star})$, 
we can then eliminate the $x$ variable and equivalently write the
dynamical system as
\begin{equation}\label{eqn:Exact_ODE_x_less}
    \begin{split}
    \dot{r}_{\star} &= p_{\star}   {,}\\
    \dot{p}_{\star} &= F(r_{\star}) -\left[\frac{1}{2}b(r_{\star})^{\top}A(r_{\star})^{-1}\frac{\partial A}{\partial r}(r_{\star})A(r_{\star})^{-1}b(r_{\star})-\frac{\partial b}{\partial r}(r_{\star})^{\top}A(r_{\star})^{-1}b(r_{\star})\right]     {.}
    \end{split}
\end{equation}

Following Eq.~\eqref{eqn:General_Stochastic}, the corresponding
Stochastic-XLMD method reads 
\begin{equation}\label{eqn:SDE}
    \begin{split}
    \dret &= \pet, \\
    \dpet &= F(\ret) - \left[\frac{1}{2}\xet^{\top}\frac{\partial A}{\partial r}(\ret)\xet - \frac{\partial b}{\partial r}(\ret)^{\top}\xet\right], \\
    \dxet &= \ve^{-1/2} \yet ,\\
    \dyet &= \ve^{-1/2}\left[b(\ret)-A(\ret)\xet\right] - \ve^{-1/2}\gamma \yet + \ve^{-1/4}\sqrt{2\gamma T} \dot{W} {,}
\end{split}
\end{equation}
where $\ve$,$\gamma$ and $T$ are positive parameters, 
and $W(t)$ is the standard Brownian motion. 
The last equation in~\eqref{eqn:SDE} is a stochastic differential
equation (SDE)
whose rigorous interpretation follows the It\^{o} integral formulation,
which can be simplified in this case as
\begin{align*}
    \yet(t) - \yet(0) &=  \ve^{-1/2}\int_0^t\left[b(\ret(s))-A(\ret(s))\xet(s)\right]ds \\
    &\quad - \ve^{-1/2}\int_0^t\gamma \yet(s)ds + \ve^{-1/4}\sqrt{2\gamma T} W(t) {.}
\end{align*}
Since we are mainly interested in the atomic dynamics, 
the initial values are assumed to be accurate, i.e. $\ret(0) = r_{\star}(0)$, 
$\pet(0) = p_{\star}(0)$.  Note that we only assume $\xet(0),\yet(0)$ are chosen
deterministically. In particular, we do not necessarily have
$\xet(0)=x_{\star}(0)$. 

%We would like to understand how and why $\ret$ and $\pet$ can give reasonable approximations for $r_{\star}$ and $p_{\star}$ under particular 
%choices of parameters $\ve$, $\gamma$, $T$, together with 
%the optimal strategy of choosing those parameters. 

If $\gamma = T = 0$, the SDE~\eqref{eqn:SDE} degenerates 
to a singularly perturbed ODE, which is exactly the XL-BOMD
approach~\eqref{eqn:General_Deterministic}.  In this case, numerical
results show that the convergence of $\re$ to $r_{\star}$ depends
sensitively on the initial value of $x(0)$.
%if the initial value $x(0)$ is not compatible, 
%then in general $\re$ will not converge to $r$ 
%no matter how the parameter $\ve$ is chosen. 
Figure~\ref{fig:comp_DAE_SPE_SDE} shows that with the inaccurate initial
guess for $x(0)$, the XL-BOMD approach gives inaccurate dynamics, 
while the Stochastic-XLMD approach gives a much more accurate
approximation (see section~\ref{sec:accuracy} for the detailed setup). Here
we plot the trajectories of the first entries of $r$ and $x$, and the total energy. 

\begin{figure}
    \centering
    \includegraphics[width=0.48\linewidth]{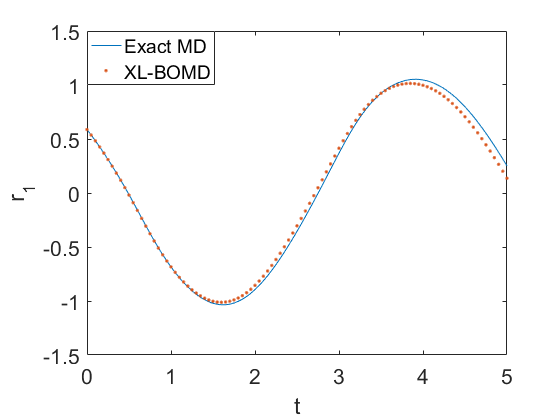}
    \includegraphics[width=0.48\linewidth]{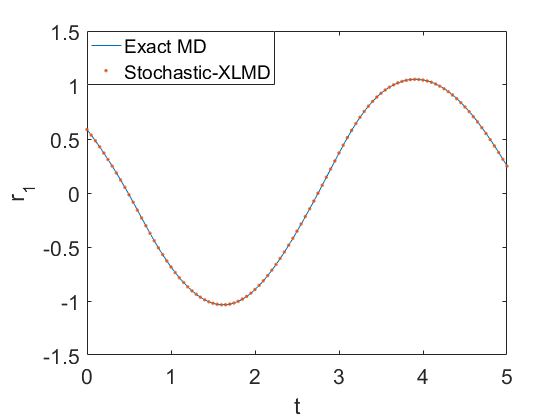}\\
    \includegraphics[width=0.48\linewidth]{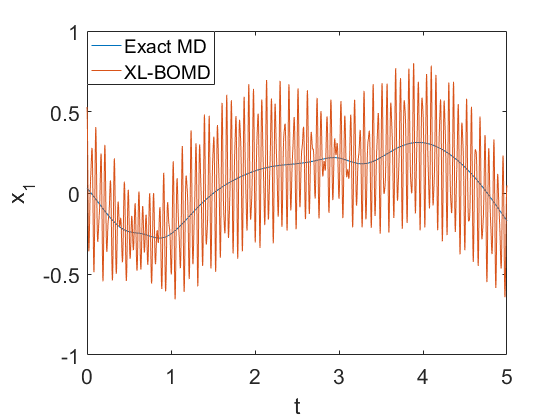}
    \includegraphics[width=0.48\linewidth]{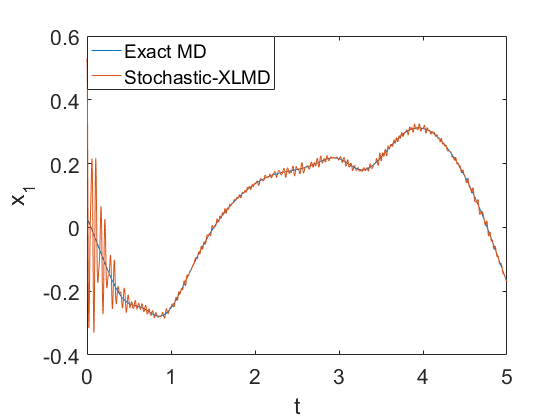}\\
    \includegraphics[width=0.48\linewidth]{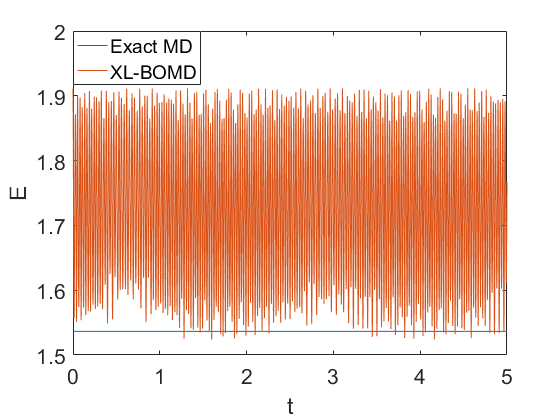}
    \includegraphics[width=0.48\linewidth]{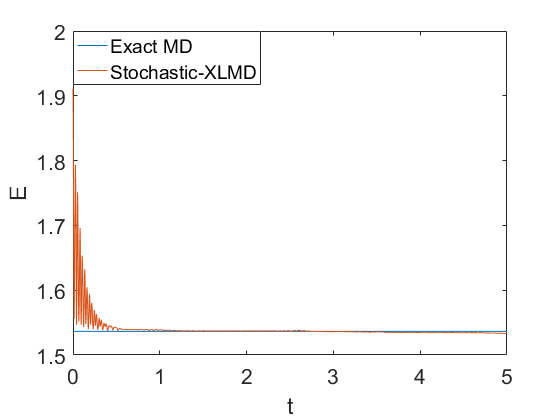}
    \caption{Left: Comparison of the exact MD and 
    XL-BOMD with $\ve = 10^{-4}$. Right: Comparison of the exact MD and
    Stochastic-XLMD with $\ve = 10^{-4}$, $\gamma = 0.100$ 
    and $T = 10^{-4}$. The three rows are the first entry of $r$, the 
    first entry of $x$ and total energy $\frac{1}{2}|p|^2+U+Q$, respectively. 
    }
    \label{fig:comp_DAE_SPE_SDE}
\end{figure}

The difference of the convergence behaviors can be explained  by
the method of time averaging in multiscale analysis.  Note that the
fast dynamics in the XL-BOMD approach is not ergodic.  In fact, the
fast dynamics in that case is a Hamiltonian ODE. Thus any smooth
function of the Hamiltonian will lie in the null space of the
corresponding generator.  Therefore, the error of the initial values
will be carried through the entire simulation.  We refer readers
to~\cite{BornemannSchutte1997} for an explicit example on how the
initial values influence the entire Hamiltonian dynamics (with strong
constraining potential).

However, in Stochastic-XLMD, the fast Langevin dynamics is
ergodic~\cite{LeimkuhlerMatthews2015,Pavliotis2014}, which means that
the stationary movement of $(x,y)$ is independent of the initial
values.  Consider the following Langevin dynamics with $\ve = 1$ and
fixed $r$ viewed as a parameter (and we omit the explicit dependence
on $r$ in notations for clarity),
\begin{equation}\label{eqn:Langevin}
     \begin{split}
          \dot{x} &= y, \\
          \dot{y} &= b-Ax-\gamma y + \sqrt{2\gamma T}\dot{W} {.}
     \end{split}
\end{equation}
The system~\eqref{eqn:Langevin} is ergodic with 
an invariant probability density
\[
     \rho_{\infty}(x,y;r) \propto \exp\left(-\frac{\vert y \vert^2}{2T}\right) 
        \exp\left( -\frac{(x-A^{-1}b)^{\top}A(x-A^{-1}b)}{2T} \right) {.}
\]
That is, as $t \rightarrow \infty$, the solution $(x,y)$ of 
Eq.~\eqref{eqn:Langevin} will converge in distribution to the invariant distribution
regardless of the initial values. 

Now we go back to Stochastic-XLMD~\eqref{eqn:SDE} and apply the method of
averaging. 
Note that the time scale of the oscillation of $\xet$ and $\yet$ 
is $\Or(\ve^{1/2})$. 
If we consider an intermediate time period $[t_1,t_2]$, 
for example, $t_2-t_1 = \Or(\ve^{1/4})$, 
then within this time period the variable $\ret$ almost remains constant, 
and the fast variable $\xet$ has already converged to 
the invariant distribution. 
Therefore when $\ve$ is small, it is reasonable to reckon that 
the slow dynamics of $\ret, \pet$ can be 
approximated by the averaged dynamics, in which the fast variable $\xet$ 
is averaged out with respect to the invariant measure. 
This can also be formally derived by the multiscale expansion
method (see for example~\cite[Chapter 10]{PavliotisStuart2008}). 

More specifically, let the averaged dynamics be defined as
\begin{equation}\label{eqn:Averaged_definition}
        \begin{split}
        \drave &= \pave {,}\\
    \dpave &= F(\rave) - \int_{\R^{2d}} \left[\frac{1}{2}x^{\top}\frac{\partial A}{\partial r}(\rave)x - \frac{\partial b}{\partial r}(\rave)^{\top}x\right] 
    \rho_{\infty}(x,y;\rave)dxdy {.}
        \end{split}
    \end{equation}
    After explicit evaluation of the integral (see the end of
    section~\ref{sec:poisson} for details), we arrive at 
\begin{equation}\label{eqn:Averaged_ODE}
    \begin{split}
        \drave &= \pave,  \\
        \dpave &= F(\rave) - \left[ \frac{1}{2}b(\rave)^{\top}A(\rave)^{-1}\frac{\partial A}{\partial r}(\rave)A(\rave)^{-1}b(\rave) - \frac{\partial b}{\partial r}(\rave)^{\top}A(\rave)^{-1}b(\rave) \right] - Tg(\rave){,}
    \end{split}
\end{equation}
where $g(r) = (g_1(r),\cdots, g_d(r))^{\top}$, 
\begin{equation}\label{eqn:Averaged_ODE_g}
        g_i(r) = \frac{1}{2}\sum_{k,l} \left(\frac{\partial A}{\partial r_i}\right)_{kl}\left(A^{-1}\right)_{kl} 
    = \frac{1}{2}\Tr\left(\frac{\partial A}{\partial r_i}(r)A^{-1}(r) \right) {.}
\end{equation}
Compare with the exact MD~\eqref{eqn:Exact_ODE_x_less}, there is only
one extra term $-Tg(\rave)$.  Therefore, we can expect that, as
$T\to 0$, the solution $(\rave, \pave)$ of~\eqref{eqn:Averaged_ODE}
converges to the exact solution $(r_{\star}, p_{\star})$, and 
$(\ret,\pet)$ converges to the exact solution $(r_{\star}, p_{\star})$
as $\ve,T\to 0$.  Since the time averaging relies on the
ergodicity of the fast dynamics, it is clear that the convergence is
independent of the initial value of the latent variables.

%In summary, Stochastic-XLMD approach~\eqref{eqn:SDE} can give reasonable approximation to 
%the exact MD~\eqref{eqn:Exact_ODE} when $\ve$ and $T$ are small. 
In order to study the efficiency of Stochastic-XLMD with respect to $\gamma$, 
first let us consider two limiting scenarios.
%numerical results indicate that the best choice of $\gamma$ should be 
%neither too large nor too small. 
%This can be understood by considering two limit. 
If $\gamma$ is very close to 0, the fast dynamics will be very close 
to the XL-BOMD dynamics, which leads to inaccurate solutions
if the initial condition of the latent variable is inaccurate. 
If $\gamma$ is very large, the noise must also increase according to
the fluctuation-dissipation relation. The fast dynamics 
then behaves as the Brownian dynamics, and thus it would take longer to reach
the invariant distribution for a fixed $r$.
We find that the optimal choice of $\gamma$ should be $\Or(1)$, and this will
be confirmed by numerical results.
%As a result, in the analysis part we will always consider 
%$\gamma$ as a fixed $\Or(1)$ parameter, and that is also the reason 
%why there is an absence in the notations of 
%solutions to the Stochastic-XLMD~\eqref{eqn:SDE}. 
%\JL{In our situation, as the dynamics is linear, so that we can write down explicitly the semigroup, the $\gamma$ dependence is quite explicit, as stated below in Proposition 5. I wonder if we shall track the dependence in the main result?}\LL{Agreed}\DA{I will track the dependence on $\gamma$ carefully. }

%At the end of this section, we make the idea of averaging precise and 
%state the main theorem, with discussions on parameter dependence. 
Now we state the main result precisely. We consider a fixed time
interval $[0,t_f]$ with $t_f$ fixed and independent of $\xi$. 
Throughout the paper we denote by $|a|$ the absolute value of $a$ if $a$ is a scalar, and the vector 2-norm of $a$ if $a$ is a vector. 
$\|A\|_2, \|A\|_F, \|A\|_{*}$ denote the matrix 2-norm, the matrix Frobenius norm 
and the matrix trace norm, respectively. 
We make the following assumptions: 
\begin{enumerate}
    \item $A:\R^{d}\ra\mathcal{S}_{++}^{d}$ is a smooth map 
    with globally bounded $\|A\|_2, \left\|\frac{\partial A}{\partial r_j}\right\|_*, \left\|\frac{\partial^2 A}{\partial r_j^2}\right\|_*$, 
    $j = 1,\cdots,d$. 
    Furthermore, there exists a constant 
    $\kappa>0$ such that $A(r)\succeq \kappa$ for all $r\in\R^{d}$. %\JL{perhaps change $C$ to a less generic notation?}
    \item $b:\R^{d}\ra\R^{d}$ is a smooth map 
    with globally bounded $|b|, \left|\frac{\partial b}{\partial r_j}\right|, 
    \left|\frac{\partial^2 b}{\partial r_j^2}\right|$, $j = 1,\cdots,d$.
    \item $F:\R^{d}\ra\R^{d}$ is a smooth map 
    with globally bounded $\left|\frac{\partial F}{\partial r_j}\right|$, $\left|\frac{\partial^2 F}{\partial r_j^2}\right|$, $j = 1,\cdots,d$. 
    \item Initial values for $(r_{\star},p_{\star})$, 
    $(\ret,\pet,\xet,\yet)$ and $(\rave,\pave)$ are deterministic, with 
    $r_{\star}(0) = \ret(0) = \rave(0)$, 
    $p_{\star}(0) = \pet(0) = \pave(0)$. 
    \item For $0 < T < 1$, $0 < \ve < 1$, $\gamma > 0$, 
    the solution $(\rave,\pave)$ 
    is bounded independently of $T$, and 
    the solution $(\ret,\pet,\xet,\yet)$ is bounded in the sense that
    \[
        \mathbf{E}\left(\sup_{0\leq t \leq t_f}\vert\xet(t)\vert^2\right), \quad \mathbf{E}\left(\sup_{0\leq t \leq t_f}\vert\yet(t)\vert^2\right) {,}   \]
    \[
        \mathbf{E}\int_0^{t_f} |\ret(s)|^4 ds, \quad \mathbf{E}\int_0^{t_f} |\pet(s)|^4 ds, \quad 
        \mathbf{E}\int_0^{t_f} |\xet(s)|^4 ds, \quad \mathbf{E}\int_0^{t_f} |\yet(s)|^4 ds
    \]
    are bounded independently of $\ve$,$T$ and $\gamma$.      
\end{enumerate}
Here the first three assumptions assure the existence and uniqueness 
of the smooth, globally bounded solutions of~\eqref{eqn:Exact_ODE} 
and~\eqref{eqn:Averaged_ODE}, together with 
the existence and uniqueness of the solution of~\eqref{eqn:SDE}. 
It is worth mentioning that weakening all the assumptions is possible by 
proving some \emph{a priori} bounds, but 
we limit ourselves to the simple setup for expository purposes. 
Throughout this paper, $C$ will denote 
a sufficiently large constant of possibly varying size, which 
is independent of $\xi$ but may depend on other constant factors 
such as the final time $t_f$ and the dimension $d$. 

\smallskip 

\begin{thm}\label{thm:Main_Theorem}
    Let $(r_{\star},p_{\star})$ solve 
    the exact MD~\eqref{eqn:Exact_ODE_x_less}
    and $(\ret,\pet,\xet,\yet)$ 
    solve the Stochastic-XLMD~\eqref{eqn:SDE}. 
    Then for any $0 <\ve < 1$, $0 < T < 1$, $\gamma > 0$
    there exists a constant $C > 0$ such that 
    \begin{align*}
        &\mathbf{E}\left(\sup_{0 \leq t \leq t_f}\vert \ret(t) -
        r_{\star}(t)\vert \vee 
        \vert \pet(t) - p_{\star}(t)\vert \right) \\
        \leq &  C\left[\left(\frac{1}{\delta_\gamma}+\frac{1}{\delta_\gamma^2}\right)\ve^{1/2} + \left(\frac{\gamma}{\delta_\gamma^2}+1\right)\ve^{1/4}T^{1/2} + T + \frac{\gamma}{\delta_\gamma^3}\ve^{1/2}T\right]{.}
%        &\mathbf{E}\left(\sup_{0 \leq t \leq t_f}\right) \leq C\left[\left(\frac{1}{\delta_\gamma}+\frac{1}{\delta_\gamma^2}\right)\ve^{1/2} + \left(\frac{\gamma}{\delta_\gamma^2}+1\right)\ve^{1/4}T^{1/2} + T + \frac{\gamma}{\delta_\gamma^3}\ve^{1/2}T\right] {,}
    \end{align*}
    where $\delta_\gamma$ is a $\gamma$-dependent positive real number defined as 
    \begin{equation}\label{eqn:gap}
        \delta_{\gamma} = \begin{cases}
            \gamma/4, \quad 0 < \gamma \leq 2\sqrt{\kappa} {,} \\
            (\gamma - \sqrt{\gamma^2 - 4\kappa})/4, 
            \quad \gamma > 2\sqrt{\kappa} {.}
        \end{cases}
    \end{equation}
%    \JL{it would be good to make a comment on the dependence of $C$ on $t_f$ (exponential due to Gronwall, I suppose?)}
%    \DA{Yes the time dependence is exponential due to Gronwall. I add a comment in the following paragraph. }
\end{thm}

Before proceeding with the proof in section~\ref{sec:proof}, several
remarks are in order.

Theorem~\ref{thm:Main_Theorem} verifies the intuition that $\ve$ and
$T$ should be small to yield a reasonable approximation, and provides the
convergence order with respect to $\ve$ and $T$. More specifically, if
we fix $\gamma$ and all other parameters such as $t_f$, then the
dominating part of errors becomes $\Or(\ve^{1/2}+\ve^{1/4}T^{1/2}+T)$,
which suggests the optimal strategy for choosing parameters should be
$T = \Or(\ve^{1/2})$. Therefore the optimal convergence order with
respect to $\ve$ is $1/2$.  The optimality of the convergence order
will be verified by numerical results in section~\ref{sec:numer}.
Theorem~\ref{thm:Main_Theorem} also suggests that $\gamma$ should not
be too large or too small.  For fixed $\ve$ and $T$, the error bounds
will go to infinity if $\gamma \rightarrow 0$ or
$\gamma \rightarrow \infty$.  We also remark that the constant $C$
depends exponentially on $t_f$ due to the use of Gronwall's
inequality.

The convergence of XL-BOMD type schemes in the linear response regime
has been studied in~\cite{LinLuShao2014}, where the energy depends
quadratically both with respect to $r$ and $x$. In such a case, the dynamics is
diagonalizable, and the convergence of XL-BOMD can be studied using
perturbation theory with respect to the eigenvalues of the diagonalized
systems. Such a strategy cannot be used for
the polarizable force field model, in which the energy is non-quadratic with 
respect to $r$, though it is quadratic with respect to $x$. 

A rigorous proof of the method of averaging for model SDEs is
given in~\cite[Chapter 17]{PavliotisStuart2008}, where the generator
of the auxiliary SDE is assumed to be a non-degenerate elliptic
operator and the domain of interest is assumed to be compact. From the
technical perspective, the key of the proof is to apply the It\^o
formula to the solution of the Poisson equation~\eqref{eqn:Poisson}. 
The aforementioned assumptions facilitates the growth estimate of the
solution of the Poisson equation corresponding to the SDE. Our proof
generalizes the method to the Stochastic-XLMD case, where the
generator of the Langevin equation is a degenerate elliptic operator,
and the domain is the whole space $\RR^d$, which requires a more careful
study of the Poisson equation~\eqref{eqn:Poisson}.

The existence and uniqueness of a smooth solution to the Poisson
equation can be assured in a more general case than the quadratic
interaction energy~\cite{LeimkuhlerMatthewsStoltz2015,HairerPavliotis2008}. 
Under proper assumptions such that the interaction energy satisfies the
Poincar\'e inequality and grows moderately (both of which the
quadratic interaction energy satisfies), the generator is invertible
within the space $\{u\in H^1(d\mu): \int u d\mu = 0\}$ where $d\mu$ is
the invariant measure.  This is a result from
hypocoercivity~\cite{Villani2009}, which focuses on the convergence to the
stationary state for certain classes of degenerate diffusive
equations. 
%This area has attracted a lot of attention recently.  
The smoothness of the solution is a straightforward result from the
hypoellipticity~\cite{Pavliotis2014}, which can be traced back to
H\"ormander~\cite{Hormander1961}.

We would also like to mention a series of
papers~\cite{PardouxVerrtennikov2001,PardouxVerrtennikov2003,PardouxVerrtennikov2005},
which provide a more general study of the solution of the Poisson
equation on $\R^d$  both for the non-degenerate case and the degenerate case.
The solution of the Poisson equation in our proof
(Eq.~\eqref{eqn:Poisson_solution}) originates
from~\cite{PardouxVerrtennikov2005}. Our work generalizes the results
of~\cite{PardouxVerrtennikov2005,Talay2002} (though for a much simpler
scenario), %\JL{note I added the previous sentence}
in the sense that
we can describe the explicit dependence of the constant on parameters
such as $\gamma,\ve,T$, which is needed for the convergence rate of
the Stochastic-XLMD scheme.
%In fact, the constant
%$C$ surely depends on $T$ if the right hand side is in a more general
%form rather than quadratic function in our setup.  This can be explained
%by the fact that the convergence speed of the Langevin process to the
%invariant distribution, measured by total variance, becomes slower as
%$T$ becomes smaller. 
%Similarly, the exponential decay of the integrand in 
%Eq.~\eqref{eqn:Poisson_solution} (expectation of an observable) 
%has also been proved in~\cite{Talay2002} 
%with no explicit dependence on $\gamma$ and $T$. 
%In our work, taking advantage of the quadratic interaction energy, 
%we proceed a more explicit and direct calculation 
%and analysis, and focus on the dependence of the errors on 
%parameters $\ve$,$T$ and $\gamma$. 

\section{Proof of the main theorem}\label{sec:proof}

In this section we prove Theorem~\ref{thm:Main_Theorem} through combining the 
 following two theorems. 

\begin{thm}\label{thm:SDE_Error}
    Let $(\rave,\pave)$ solve 
    the averaged dynamics~\eqref{eqn:Averaged_ODE}
    and $(\ret,\pet,\xet,\yet)$ solve the Stochastic-XLMD~\eqref{eqn:SDE}. 
    Then for any $0 <\ve < 1$, $0 < T < 1$, $\gamma > 0$
    there exists a constant $C > 0$ such that 
    \begin{multline*}
      \mathbf{E}\left(\sup_{0 \leq t \leq t_f}\vert \ret(t) -
        \rave(t)\vert \vee \vert \pet(t) - \pave(t)\vert\right) \\
        \leq C\left[\left(\frac{1}{\delta_\gamma}+\frac{1}{\delta_\gamma^2}\right)\ve^{1/2} + \left(\frac{\gamma}{\delta_\gamma^2}+1\right)\ve^{1/4}T^{1/2} + \frac{\gamma}{\delta_\gamma^3}\ve^{1/2}T\right].
%        &\mathbf{E}\left(\sup_{0 \leq t \leq t_f}\vert \pet(t) - \pave(t)\vert\right) \leq C\left[\left(\frac{1}{\delta_\gamma}+\frac{1}{\delta_\gamma^2}\right)\ve^{1/2} + \left(\frac{\gamma}{\delta_\gamma^2}+1\right)\ve^{1/4}T^{1/2} + \frac{\gamma}{\delta_\gamma^3}\ve^{1/2}T\right] {.}
    \end{multline*}
\end{thm}

\begin{thm}\label{thm:Averaged_Error}
    Let $(r_{\star},p_{\star})$ solve 
    the exact dynamics~\eqref{eqn:Exact_ODE_x_less}
    and $(\rave,\pave)$ solve 
    the averaged dynamics~\eqref{eqn:Averaged_ODE}. 
    Then for any $0 < T < 1$, 
    there exists a constant $C > 0$ such that 
    \[
    \sup_{0 \leq t \leq t_f}\vert \rave(t) - r_{\star}(t)\vert \vee
    \vert \pave(t) - p_{\star}(t)\vert \leq CT.
    \]
%    \begin{align*}
%        & {,}\\
%        &\sup_{0 \leq t \leq t_f}\vert \pave(t) - p_{\star}(t)\vert \leq CT {.}
%    \end{align*}
\end{thm}

Note that these two theorems describe two different contributions to the error, 
and the combination of them directly implies
Theorem~\ref{thm:Main_Theorem}.

%In order to prove Theorem~\ref{thm:SDE_Error}, we follow the method
%in~\cite[Chapter 17]{PavliotisStuart2008}, which establishes
%the convergence of time averaging method under non-degenerate
%uniformly elliptic generator on a compact domain.  In our situation, we
%have a degenerate elliptic operator on the non-compact domain $\R^d$,
%and hence the proof in~\cite{PavliotisStuart2008} is not directly
%applicable. The key to the proof is to estimate the growth of the
%solution to a Poisson equation. Hence 

Theorem~\ref{thm:Averaged_Error} is a direct result from 
the theorem of Alekseev and Gr\"obner~\cite[Theorem 14.5]{HairerNorsettWanner1987}. 
In order to prove Theorem~\ref{thm:SDE_Error}, we generalize the method
in~\cite[Chapter 17]{PavliotisStuart2008}, where the key is 
to apply the It\^o formula to the solution of the Poisson equation corresponding to 
Langevin dynamics. 
The rest of the proof is organized as follows. 
In section~\ref{sec:langevin} we first record some useful properties of the
Langevin dynamics~\eqref{eqn:Langevin}.  We then discuss the 
solution of the Poisson equation in section~\ref{sec:poisson}.
The proof of Theorem~\ref{thm:SDE_Error} and~\ref{thm:Averaged_Error}
follows in section~\ref{sec:proof_remain}.

\subsection{Properties of Langevin Dynamics}\label{sec:langevin}

We first study the linear SDE~\eqref{eqn:Langevin} with
fixed $r$, of which the solution can be obtained explicitly.
We remark that despite the $r$ dependence in $A$ and $b$, 
the bounds of $b$ and $A^{-1}$ are independent of $r$ by assumption 1 and 2. 

We start with the standard ergodic property of Langevin dynamics. The proof of Proposition~\ref{prop:Langevin} 
can be found in e.g. \cite[Prop. 6.1 and section 3.7]{Pavliotis2014}.
\begin{prop}\label{prop:Langevin}
    Let $(x(t),y(t))$ denote the solution of SDE~\eqref{eqn:Langevin}. 
    Let 
    \[
        \mathfrak{B} = \left(\begin{array}{cc}
            0 & -I_d \\
            A & \gamma I_d
        \end{array}\right), \quad 
        \mathfrak{z}(t) = \left(\begin{array}{c}
        x(t) - A^{-1}b \\
        y(t)
    \end{array}\right). 
    \]
    Then 
    
    (a) Let $\mathcal{L}_0$ be the generator of the Langevin dynamics~\eqref{eqn:Langevin}:
    \begin{equation}\label{eqn:Langevin_Generator}
        \mathcal{L}_0\psi =  y \cdot \nabla_{x}\psi + 
    (b - Ax - \gamma y)\cdot \nabla_{y} \psi
    + \gamma T \Delta_{y} \psi {.}
    \end{equation}
    and the adjoint of $\mathcal{L}_0$ is denoted $\mathcal{L}_0^{*}$.  
    Then the probability density function of $\mathfrak{z}(t)$ 
    is the solution of the Fokker-Planck equation
    \begin{equation}\label{eqn:Langevin_Fokker_Planck}
        \frac{d}{dt}\rho_t = \mathcal{L}_0^{*} \rho_t.
    \end{equation}
    Furthermore, the density function is explicitly given by
    \[
    \rho_t(\mathfrak{z}) = \frac{1}{Z_t} \exp \left[-\frac{1}{2} \left(\mathfrak{z}-e^{-\mathfrak{B}t}\mathfrak{z}(0)\right)^{\top} \mathfrak{S}_t^{-1} \left(\mathfrak{z}-e^{-\mathfrak{B}t}\mathfrak{z}(0)\right) \right] {,}
    \]
    where $\mathfrak{S}_t$ is given by
    \begin{equation}\label{eqn:Langevin_Sigmat}
        \mathfrak{S}_t = \int_0^t e^{-\mathfrak{B}s}
        \left(\begin{array}{cc}
            0 & 0 \\
            0 & 2\gamma T I_d
        \end{array}\right)
        e^{-\mathfrak{B}^{\top}s}ds
    \end{equation}
    and $Z_t$ is the normalization constant 
    \[
        Z_t = (2\pi)^d \sqrt{\det \mathfrak{S}_t} {.}
    \]
    
    (b) The SDE~\eqref{eqn:Langevin} is ergodic. 
    
    (c) There exists a unique invariant density $\rho_{\infty} (x,y;r) \in C^{\infty}(\R^d\times \R^d;\R^d)$ 
    such that 
    \[
        \mathcal{L}_0^{*} \rho_{\infty} = 0, \quad \rho_{\infty} > 0{.}
    \]
    Furthermore, the invariant measure is a Gaussian distribution in
    terms of $\mathfrak{z}$, which is mean-zero and its covariance
    matrix $\mathfrak{S}_{\infty}$ is defined
    by~\eqref{eqn:Langevin_Sigmat} after taking the limit $t\to \infty$.  Equivalently, in terms of $x$ and
    $y$, the invariant density $\rho^{\infty}$ is given by
    \[
        \rho_{\infty} = \frac{1}{Z_{\infty}}\exp\left(-\frac{\vert y \vert^2}{2T}\right) 
        \exp\left( -\frac{(x-A^{-1}b)^{\top}A(x-A^{-1}b)}{2T} \right) {,}
    \]
where $Z_{\infty}$ is the normalization constant 
\[
    Z_{\infty} = (2\pi)^{d}T^{d}\left(\sqrt{\det A}\right)^{-1}{.}
\]
\end{prop}

The convergence rate of the covariance matrix $\mathfrak{S}_t$ towards
$\mathfrak{S}_{\infty}$ is recorded in
Proposition~\ref{prop:Langevin_conv_rate}.  

\begin{prop}\label{prop:Langevin_conv_rate}
  %Assume that the eigenvalues of matrix $A$ are  
     % $\lambda_1 \geq \cdots \geq \lambda_d > 0$. 
  Let $\delta_\gamma$ denote the positive real number defined in
  Eq.~\eqref{eqn:gap}.  Then there exists a constant $C>0$ such that
  \begin{align*} 
    \mathrm{(a)} \qquad & 
        \|e^{-\mathfrak{B}t}\|_{2} \leq C e^{-\delta_\gamma t}{,} \\
    \mathrm{(b)} \qquad & 
        \|\mathfrak{S}_t - \mathfrak{S}_{\infty} \|_{2}
                        \leq C\frac{\gamma}{\delta_\gamma}Te^{-2\delta_\gamma t}.
  \end{align*}
\end{prop}

\begin{proof} 
See Appendix~\ref{append:proof_Prop5}. 

\end{proof}

\subsection{Poisson Equation}\label{sec:poisson}

%\JL{I think we need some motivation here; perhaps add a sketch of the proof after the main theorem and explain the elements of the proof?} 
%\DA{I modified the paragraphs after Theorem 3 to give some motivation of this section. }
Define
\begin{equation}\label{eqn:ODE_RHS}
    h(r,x) := F(r)-\left(\frac{1}{2}x^{\top}\frac{\partial A}{\partial r}x-\left(\frac{\partial b}{\partial r}\right)^{\top}x\right){.}
\end{equation}
We are interested in the following Poisson equation corresponding to the
Langevin dynamics.
\begin{equation}\label{eqn:Poisson}
        \begin{split}
        &\mathcal{L}_0 \phi(x,y;r) = h(r,x) - \int_{\R^{2d}}h(r,x')\rho_{\infty}(x',y';r)dy'dx' {,} \\
        &\int_{\R^{2d}}\phi(x,y;r)\rho_{\infty}(x,y;r)dydx = 0 {.}
        \end{split}
\end{equation}

\begin{prop}\label{prop:Poisson}
  For any $0 < T < 1$, $\gamma > 0$, there exists a smooth function
  $\phi(x,y;r)$ which solves the Poisson equation~\eqref{eqn:Poisson}
  and satisfies the estimates %\JL{is the one for $\nabla_r \phi$ also Frobenius norm?}
    %\DA{The bound of $\nabla_r \phi$ is used in estimating $\theta$ (6th equation in section 3.3), and I think 2-norm is sufficient there. }
    \begin{equation}\label{eqn:Poisson_Estimate}
    \begin{split}
        |\phi(r,x,y)|  
        %\LL{\phi(x,y;r). \text{same below. Also this should be a norm}} 
        &\leq C\left[\frac{\gamma}{\delta_\gamma^2}T+ \frac{1}{\delta_\gamma}(1+|x|^2+|y|^2)\right] \\
        \|\nabla_{(x,y)}\phi(r,x,y)\|_F &\leq C\frac{1}{\delta_\gamma}(1 + |x| + |y|)\\
        \|\nabla_r\phi(r,x,y)\|_2 &\leq C\left[ \frac{\gamma}{\delta_\gamma^2}T + \frac{1}{\delta_\gamma}(1+|x|^2+|y|^2) + \frac{\gamma }{\delta_\gamma^3}T + \frac{1}{\delta_\gamma^2}(1+|x|^2+|y|^2)\right] {,}
    \end{split}
  \end{equation}
  where $C$ is a positive constant which is independent of $\gamma, T, r,x,y$. %\LL{should we clarify that by $\|\nabla_r\phi(r,x,y)\|_2$ we really mean the operator norm of the matrix $\nabla_r\phi$? As seen from the comment above this quantity can easily be mistaken as a vector.} 
  %\JL{perhaps we shall say here what $C$ depends on}
  %\DA{I add a sentence saying what $C$ does not depend on, I think maybe it is a little clearer. }
\end{prop}

\begin{proof}
  The proof is constructive. 
    Let 
    \[
        f(x,y;r) = h(r,x) - \int_{\R^{2d}}h(r,x')\rho_{\infty}(x',y';r)dy'dx' {.}
    \]
    Define
    \[
    v(x,y,t;r) = \mathbf{E}_{x,y}f(x(t),y(t);r){,}
    \]
    and
    \begin{equation}\label{eqn:Poisson_solution}
    \begin{split}
        \phi(x,y;r) &= -\int_0^{\infty} v(x,y,s;r)ds \\
        &= - \int_0^{\infty} \left[\mathbf{E}_{x,y} h(r,x(s)) - \int_{\R^{2d}}h(r,x')\rho_{\infty}(x',y';r)dy'dx'\right] ds {.}
    \end{split}
    \end{equation}
    Here $\mathbf{E}_{x,y}$ means the expectation with respect to $(x(t),y(t))$, which is the solution to 
    the SDE~\eqref{eqn:Langevin} with initial values 
    $x(0) = x, y(0) = y$. 
    %Here $(x(t),y(t))$ %\JL{why putting in this way?} 
    %is the solution to 
    %the SDE~\eqref{eqn:Langevin} with initial values 
    %$x(0) = x, y(0) = y$. 
    %\JL{I think we shall rather explain the notation $\mathbf{E}_{x,y}$ here instead}
%    We will show that such $\phi$ is well-defined 
%    and satisfies all the desired properties. 

    We organize the proof in a few steps below.
    
    (1) $\phi$ is well-defined. The key observation is that $h$ is a
    quadratic function in $x$. Hence $\mathbf{E}_{x,y}h(r,x(s))$ can
    be computed explicitly as $x(s)$ is a Gaussian random variable by
    Proposition~\ref{prop:Langevin}.  Specifically, we still use the
    notations in Proposition~\ref{prop:Langevin} and let
    $\mathfrak{D}(s)$ denote the top $d$ rows of the matrix
    $e^{-\mathfrak{B}s}$, then
    \begin{align*}
         \mathbf{E}_{x,y} h_k(r,x(s)) 
         &= \mathbf{E}_{x,y} \left[F_k(r)-\left(\frac{1}{2}x(s)^{\top}\frac{\partial A}{\partial r_k}x(s)-\frac{\partial b}{\partial r_k}^{\top}x(s)\right)\right] \\
         &= F_k(r) - \mathbf{E}_{x,y}\left(\frac{1}{2}\left(x(s)-\mathbf{E}_{x,y} x(s)\right)^{\top}\frac{\partial A}{\partial r_k}\left(x(s)-\mathbf{E}_{x,y} x(s)\right)\right) \\
         & \quad - \frac{1}{2}\mathbf{E}_{x,y}x(s)^{\top}\frac{\partial A}{\partial r_k} \mathbf{E}_{x,y}x(s)
         +\frac{\partial b}{\partial r_k}^{\top}\mathbf{E}_{x,y}x(s) \\
         &= F_k(r) - \frac{1}{2}\Tr\left(\frac{\partial A}{\partial r_k}\mathfrak{S}_t^{11} \right)
         - \frac{1}{2}\left(A^{-1}b + \mathfrak{D}(s)\mathfrak{z}\right)^{\top}\frac{\partial A}{\partial r_k} \left(A^{-1}b 
         +\mathfrak{D}(s)\mathfrak{z}\right) \\
         &\quad + \frac{\partial b}{\partial r_k}^{\top} \left(A^{-1}b 
         +\mathfrak{D}(s)\mathfrak{z}\right),
    \end{align*}
    where $\mathfrak{S}_t^{11}$ is the upper-left $d\times d$ block matrix 
    of $\mathfrak{S}_t$. 
    
    The second part of
     the integrand in Eq.~\eqref{eqn:Poisson_solution} is the expectation with respect
      to $x', y'$ with density $\rho_{\infty}$, which can be computed
      as
   % \JL{Suggested change: The second part of
    %  Eq.~\eqref{eqn:Poisson_solution} is the expectation with respect
    %  to $x', y'$ with density $\rho_{\infty}$, which can be computed
    %  as}
    \begin{multline}\label{eqn:Averaged_RHS_Computation}
      \int_{\R^{2d}}h_k(r,x')\rho_{\infty}(x',y';r)dy'dx' \\
        = F_k(r) - \frac{1}{2}\Tr\left(\frac{\partial A}{\partial r_k}\mathfrak{S}_{\infty}^{11} \right)
         - \frac{1}{2}\left(A^{-1}b\right)^{\top}\frac{\partial A}{\partial r_k} A^{-1}b 
         + \frac{\partial b}{\partial r_k}^{\top} A^{-1}b {.}
    \end{multline}
    
    The integrand $v$ in~\eqref{eqn:Poisson_solution} can be hereby
    rewritten as
    \begin{equation}\label{eqn:Poisson_Integrand}
        \begin{split}
            &\quad v_k(x,y,s;r) \\
        &= -\frac{1}{2}\Tr\left[\frac{\partial A}{\partial r_k}(\mathfrak{S}_t^{11} - \mathfrak{S}_{\infty}^{11})\right] - \frac{1}{2} \mathfrak{z}^{\top} \mathfrak{D}(s)^{\top}\frac{\partial A}{\partial r_k} A^{-1}b 
        - \frac{1}{2} b^{\top}A^{-1}\frac{\partial A}{\partial r_k}\mathfrak{D}(s)\mathfrak{z} \\
        &\quad - \frac{1}{2} \mathfrak{z}^{\top} \mathfrak{D}(s)^{\top}
        \frac{\partial A}{\partial r_k} \mathfrak{D}(s)\mathfrak{z} 
        + \frac{\partial b}{\partial r_k}^{\top} \mathfrak{D}(s)\mathfrak{z} {.}
        \end{split}
    \end{equation}
    By assumptions, $\norm{\partial A/\partial r_k}_*$,
    %\LL{should be $\norm{\partial A/\partial r_k}_*$. otherwise the norm is not defined}
    $|b|$ and $\norm{A^{-1}}_{2}$ are bounded 
    independently of $r$, and Proposition~\ref{prop:Langevin_conv_rate} 
    states that 
    $\norm{\mathfrak{S}_t^{11} - \mathfrak{S}_{\infty}^{11}}_{2}$ is bounded by $\frac{\gamma}{\delta_\gamma}Te^{-2\delta_\gamma t}$ and 
    $\norm{\mathfrak{D}(s)}_2$ is bounded by $\exp(-\delta_\gamma s)$. 
    Hence there exists a constant $C>0$ which is independent of
    $x,y,r,\gamma$ and $T$ such that 
    \[
    \Big\vert \Tr\left[\frac{\partial A}{\partial r_k}(\mathfrak{S}_t^{11} - \mathfrak{S}_{\infty}^{11})\right] \Big\vert \le 
    \Big\Vert \frac{\partial A}{\partial r_k}\Big\Vert_* \norm{\mathfrak{S}_t^{11} - \mathfrak{S}_{\infty}^{11}}_{2}
    \le C\frac{\gamma}{\delta_\gamma}Te^{-2\delta_\gamma t}.
    \]
    We may use the operator norm to bound the other terms and have, 
    \begin{equation}\label{eqn:Poisson_Exp_Decay}
    \begin{split}
        &\quad \vert v_k(x,y,s;r)\vert \\
        &\leq C\left[\frac{\gamma}{\delta_\gamma}Te^{-2\delta_\gamma s} + e^{-\delta_\gamma s}
        (1+|x|+|y|) + e^{-2\delta_\gamma s}(1+|x|^2+|y|^2)\right] \\
        &\leq C\left[\frac{\gamma}{\delta_\gamma}Te^{-2\delta_\gamma s}+ e^{-\delta_\gamma s}(1+|x|^2+|y|^2)\right]{.}
    \end{split}
    \end{equation}
    For fixed $x,y$, the integrand decays exponentially in time, and thus 
    $\phi$ is well defined. 
    
    (2) $\phi$ is a smooth solution to the Poisson equation.  The
    smoothness directly follows from the computation above. The
    mean-zero condition with respect to $\rho_{\infty}$ is straightforward from the definition of
    $\phi$. The result that $\phi$ satisfies the Poisson equation is standard 
    from the Kolmogorov backward equation. 

    %\JL{not sure if we need to include this; it is really standard}
    %\DA{How about changing the previous sentence to ``The result that $\phi$ satisfies the Poisson equation is standard from the Kolmogorov backward equation'' and removing the following part? }
    
    %Apply the It\^{o} formula to $f(x(t),y(t);r)$ and take 
    %the expectation, we obtain
    %\begin{equation}
    %    \begin{split}
    %        \frac{\partial v}{\partial t} &= \mathcal{L}_0 v,\\
    %        v(x,y,0;r) &= f(x,y;r) {.}
    %    \end{split}
    %\end{equation}
    %%This is exactly the Kolmogorov backward equation. 
    %Then we have 
    %\begin{align*}
    %    \mathcal{L}_0 \phi &= - \int_0^{\infty} \mathcal{L}_0v(x,y,s;r)ds \\
    %    &= - \lim_{S\rightarrow \infty} \int_0^{S} \mathcal{L}_0v(x,y,s;r)ds\\
    %    &= - \lim_{S\rightarrow \infty} \int_0^{S} \frac{\partial v}{\partial %t}(x,y,s;r)ds \\
    %    &= v(x,y,0;r) - \lim_{S\rightarrow \infty} v(x,y,S;r) \\
    %    &= f(x,y;r) {.}
    %\end{align*}
    %Here the last equality comes from Eq.~\eqref{eqn:Poisson_Exp_Decay}. 
    
    (3) $\phi$ allows the estimates~\eqref{eqn:Poisson_Estimate}. 
    In fact the first estimate directly follows from integrating~\eqref{eqn:Poisson_Exp_Decay} and %\LL{$|\phi|$ should be a norm}
    \[
    |\phi| \leq C\left[\frac{\gamma}{\delta_\gamma^2}T+ \frac{1}{\delta_\gamma}(1+|x|^2+|y|^2)\right] {.}
    \]
    Furthermore, $\phi$ is a quadratic function 
    of $x$ and $y$, then 
    \[
        \| \nabla_{(x,y)}\phi\|_F \leq C\left[\frac{1}{\delta_\gamma}(1 + |x| + |y|)\right] {.}
    \]
    
    In order to estimate $\nabla_r \phi$, we need to first estimate $\nabla_r \mathfrak{D}(s)$ 
    and $\nabla_r \mathfrak{S}_t^{11}$. This can be done by applying the 
    following formula~\cite{Wilcox1967} 
    \[
        \frac{d}{dt}e^{X(t)} = \int_0^1 e^{\beta X(t)}\frac{dX(t)}{dt}e^{(1-\beta)X(t)}d\beta {.}
    \]
    We have
    \begin{align*}
        \left\|\frac{\partial }{\partial r_k} \mathfrak{D}(s)\right\|_2 
        &\leq \left\|\frac{\partial }{\partial r_k}e^{-\mathfrak{B}s}\right\|_2 \\
        &= C\left\|\int_0^1 e^{-\beta\mathfrak{B}s}\frac{\partial (-\mathfrak{B}s)}{\partial r_k} e^{-(1-\beta)\mathfrak{B}s}d\beta\right\|_2 \\
        &\leq Cs\int_0^1 \left\|e^{-\beta\mathfrak{B}s}\right\|_2 \left\|\frac{\partial \mathfrak{B}}{\partial r_k}\right\|_2 \left\|e^{-(1-\beta)\mathfrak{B}s}\right\|_2 d\beta \\
        &\leq Cs\int_0^1 e^{-\beta \delta_\gamma s}e^{-(1-\beta)\delta_\gamma s} d\beta \\
        &= Cse^{-\delta_\gamma s} {,}
    \end{align*}
    and 
    \begin{align*}
        \left\|\frac{\partial }{\partial r_k} (\mathfrak{S}_t^{11} - \mathfrak{S}_{\infty}^{11})\right\|_2 
        &\leq  \left\|\frac{\partial }{\partial r_k} (\mathfrak{S}_t - \mathfrak{S}_{\infty})\right\|_2 \\
        &= C\left\|\frac{\partial }{\partial r_k}\int_t^{\infty} e^{-\mathfrak{B}s}
        \left(\begin{array}{cc}
            0 & 0 \\
            0 & 2\gamma T I_d
        \end{array}\right)
        e^{-\mathfrak{B}^{\top}s}ds\right\|_2 \\
        &\leq C\left\|\int_t^{\infty} \frac{\partial }{\partial r_k}(e^{-\mathfrak{B}s})
        \left(\begin{array}{cc}
            0 & 0 \\
            0 & 2\gamma T I_d
        \end{array}\right)
        e^{-\mathfrak{B}^{\top}s}ds\right\|_2 \\ 
        & \quad + C\left\|\int_t^{\infty} e^{-\mathfrak{B}s}
        \left(\begin{array}{cc}
            0 & 0 \\
            0 & 2\gamma T I_d
        \end{array}\right)
        \frac{\partial }{\partial r_k}(e^{-\mathfrak{B}^{\top}s)}ds\right\|_2 \\
        & \leq C\gamma T \int_t^{\infty}\left\|\frac{\partial }{\partial r_k}(e^{-\mathfrak{B}s})\right\|_2\left\|e^{-\mathfrak{B}^{\top}s}\right\|_2ds \\
        &\quad + C\gamma T \int_t^{\infty}\left\|e^{-\mathfrak{B}s}\right\|_2\left\|\frac{\partial }{\partial r_k}(e^{-\mathfrak{B}^{\top}s)}\right\|_2ds \\
        & \leq C\gamma T \int_t^{\infty} se^{-2\delta_\gamma s}ds \\
        & \leq C\gamma T \left(\frac{1}{\delta_\gamma}te^{-2\delta_\gamma t} + \frac{1}{\delta_\gamma^2}e^{-2\delta_\gamma t}\right)  {.}
    \end{align*}
    Then Eq.~\eqref{eqn:Poisson_Integrand} indicates
    \begin{align*}
        & \quad \left|\frac{\partial }{\partial r_k} v_j(x,y,s;r)\right| \\
        &\leq C\left[ \frac{\gamma}{\delta_\gamma}Te^{-2\delta_\gamma s} + e^{-\delta_\gamma s}
        (1+|x|+|y|) + e^{-2\delta_\gamma s}(1+|x|^2+|y|^2)\right] \\
        & \quad + C\left[\gamma T \left(\frac{1}{\delta_\gamma}se^{-2\delta_\gamma s} + \frac{1}{\delta_\gamma^2}e^{-2\delta_\gamma s}\right) + se^{-\delta_\gamma s}(1+|x|+|y|) + se^{-2\delta_\gamma s}(1+|x|^2+|y|^2) \right] \\
        & \leq C\left[ \frac{\gamma}{\delta_\gamma}Te^{-2\delta_\gamma s} + e^{-\delta_\gamma s}
        (1+|x|^2+|y|^2) \right] \\
        & \quad + C\left[\gamma T \left(\frac{1}{\delta_\gamma}se^{-2\delta_\gamma s} + \frac{1}{\delta_\gamma^2}e^{-2\delta_\gamma s}\right) + se^{-\delta_\gamma s}(1+|x|^2+|y|^2)\right] {.}
    \end{align*}
    Integrate  with respect to $s$ and we get 
    \begin{align*}
        \|\nabla_r \phi\|_2 \leq C\left[ \frac{\gamma}{\delta_\gamma^2}T + \frac{1}{\delta_\gamma}(1+|x|^2+|y|^2) + \frac{\gamma }{\delta_\gamma^3}T + \frac{1}{\delta_\gamma^2}(1+|x|^2+|y|^2)\right] {.}
    \end{align*}
\end{proof}

Note that in the proof of Proposition~\ref{prop:Poisson}, we have
already computed $\int h\rho_{\infty}$ in
Eq.~\eqref{eqn:Averaged_RHS_Computation}.  This is exactly the average
of the right hand side of~\eqref{eqn:SDE} with respect to the
invariant measure of the fast variables $x$ and $y$, and we obtain an
explicit formulation of the averaged dynamics. Therefore the averaged
equation defined as Eq.~\eqref{eqn:Averaged_definition} can be
equivalently given as Eq.~\eqref{eqn:Averaged_ODE}.

\subsection{Proof of Theorem~\ref{thm:SDE_Error}
and~\ref{thm:Averaged_Error}}\label{sec:proof_remain}

Since we have already obtained estimates of the solution to the Poisson
equation for the degenerate Langevin generator, we can generalize the method
in~\cite{PavliotisStuart2008} to prove Theorem~\ref{thm:SDE_Error}.

\begin{proof}[Theorem~\ref{thm:SDE_Error}]
    Notice that the generator for~\eqref{eqn:SDE} is 
    \[
        \mathcal{L} = \frac{1}{\sqrt{\ve}}\mathcal{L}_0 
        + \mathcal{L}_1
    \]
    where $\mathcal{L}_0$ is given in~\eqref{eqn:Langevin_Generator} and
    \[
        \mathcal{L}_1 = p\cdot \nabla_r + h(r,x)\cdot \nabla_p {.}
    \]
    Now we apply the It\^{o} formula to $\phi(\xet,\yet;\ret)$ 
    and obtain
    \begin{equation*}
    \begin{split}
        \frac{d\phi}{dt}(\xet,\yet;\ret) 
        &= \frac{1}{\sqrt{\ve}}\mathcal{L}_0\phi(\xet,\yet;\ret) 
        + \nabla_r\phi(\xet,\yet;\ret)\pet\\ 
        &\quad + \frac{\sqrt{2\gamma T}}{\ve^{1/4}} \nabla_y\phi(\xet,\yet;\ret) \frac{dW}{dt} {.}
    \end{split}
    \end{equation*}
    Let us introduce the notation 
    \[
    \Bar{h}(r) = F(r)-\left(\frac{1}{2}b^{\top}A^{-1}\frac{\partial A}{\partial r}A^{-1}b-b^{\top}A^{-1}\frac{\partial b}{\partial r}\right)(r){.}
    \]
    Notice that $\phi$ is the solution to the Poisson equation~\eqref{eqn:Poisson}, 
    and we obtain 
    \begin{align*}
        \frac{d\pet}{dt} &= h(\ret,\xet) \\
        &= \Bar{h}(\ret) + Tg(\ret) + \mathcal{L}_0\phi(\xet,\yet;\ret) \\
        &= \Bar{h}(\ret) + Tg(\ret) + \ve^{1/2}\frac{d\phi}{dt}(\xet,\yet;\ret) \\
        &\quad - \ve^{1/2} (\nabla_r\phi(\xet,\yet;\ret)) \pet 
        - \ve^{1/4}\sqrt{2\gamma T} \nabla_y\phi(\xet,\yet;\ret) \frac{dW}{dt} {.}
    \end{align*}
    We define 
    \begin{equation*}
    \begin{split}
               \theta(t) &= \phi(\xet(t),\yet(t);\ret(t)) 
       - \phi(\xet(0),\yet(0);\ret(0)) \\
       &\quad - \int_{0}^{t}
       (\nabla_r\phi(\xet(s),\yet(s);\ret(s)))\pet(s)ds,  
    \end{split}
    \end{equation*}
    and the martingale term 
    \[
        M(t) = - \int_{0}^{t}\sqrt{2\gamma} \nabla_y\phi(\xet(s),\yet(s);\ret(s)) dW(s) {.} 
    \]
    Then we have
    \[
        \pet(t) = \pet(0) + 
        \int_0^t \left[\Bar{h}(\ret(s)) + Tg(\ret(s))\right]ds 
        + \ve^{1/2}\theta(t) + \ve^{1/4}T^{1/2}M(t) {.}
    \]
    If we compare this with the averaged equation~\eqref{eqn:Averaged_ODE} 
    \[
        \pave(t) = \pave(0) + \int_0^t \left[\Bar{h}(\rave(s)) + Tg(\rave(s))\right]ds{,}
    \]
    and use the initial condition $\pet(0) = \pave(0)$, then we have
    \begin{equation*}
    \begin{split}
      \pet(t) - \pave(t) &= 
    \int_0^t \left[\Bar{h}(\ret(s)) - \Bar{h}(\rave(s)) + 
    Tg(\ret(s)) - Tg(\rave(s))\right]ds \\
        &\quad + \ve^{1/2}\theta(t) + \ve^{1/4}T^{1/2}M(t) {.}
    \end{split}
    \end{equation*}
    For $r$ we simply have
    \[
        \ret(t) - \rave(t) = \int_0^t \left[\pet(s) - \pave(s) \right]ds {.}
    \]
    Define the error function 
    \[
        \mathfrak{e}(t) = \left(\begin{array}{c}
            \ret(t) - \rave(t) \\
            \pet(t) - \pave(t) 
        \end{array}\right) 
    \]
    and the Lipschitz constant 
    \[L = \max\left\{1, \sup_{r\in \RR^d}\left\|\frac{\partial\Bar{h}}{\partial r}\right\|_2 + \sup_{r\in \RR^d} \left\| \frac{\partial g}{\partial r}\right\|_2\right\}{.}\] 
    Then for any $t \in [0,t_f]$, 
    \[
        \vert \mathfrak{e}(t) \vert \leq L\int_0^t \vert \mathfrak{e}(s) \vert ds 
        + \ve^{1/2}\vert \theta(t) \vert 
        + \ve^{1/4} T^{1/2}\vert M(t) \vert{.} 
    \]
    By Proposition~\ref{prop:Poisson}, we obtain 
    \begin{equation*}
    \begin{split}
    \sup_{0\leq t \leq t_f}\vert \theta(t) \vert &\leq  
    C\frac{\gamma}{\delta_\gamma^2}T  +  C\frac{1}{\delta_\gamma} + C\frac{1}{\delta_\gamma}\sup_{0\leq t \leq t_f}(\vert\xet(t)\vert^2 + \vert\yet(t)\vert^2) \\
    &\quad +  C\left(\frac{\gamma}{\delta_\gamma^2}+\frac{\gamma}{\delta_\gamma^3}\right)T\int_0^{t_f}\vert \pet(s) \vert ds \\
    &\quad + C\left(\frac{1}{\delta_\gamma}+\frac{1}{\delta_\gamma^2}\right)\int_0^{t_f}\vert \pet(s) \vert(1+|\xet(s)|^2+|\yet(s)|^2) ds
    \end{split}
    \end{equation*}
    
    and
    \[
        \mathbf{E} \left(\sup_{0\leq t \leq t_f}\vert \theta(t) \vert\right)
        \leq C\left(\frac{\gamma}{\delta_\gamma^2}T + \frac{\gamma}{\delta_\gamma^3}T+\frac{1}{\delta_\gamma}+\frac{1}{\delta_\gamma^2}\right) {.}
    \]
    For the martingale term, the It\^{o} isometry gives 
    \[
        \mathbf{E}\vert \left< M\right>_t\vert^2 
        \leq C\int_0^t 2\gamma \mathbf{E}\| \nabla_y\phi(\ret(s),\xet(s),\yet(s))\|_F^2ds 
        \leq C\frac{\gamma}{\delta_\gamma^2} {,}
    \]
    where $\left<M\right>_t$ is the quadratic variation of the martingale 
    $M(t)$ (Definition 3.18, \cite{PavliotisStuart2008}). 
    By taking expectation of the inequality 
    $\vert \left< M\right>_t\vert^{1/2} \leq \vert \left< M\right>_t\vert^{2} + 1$, we have 
    \[
        \mathbf{E}\vert \left< M\right>_t\vert^{1/2} \leq C\frac{\gamma}{\delta_\gamma^2} + 1{.}
    \]
    Hence, by the Burkholder-Davis-Gundy inequality (Theorem 3.22, \cite{PavliotisStuart2008}), 
    we obtain
    \begin{align*}
        \mathbf{E}\left(\sup_{0\leq t' \leq t}\vert \mathfrak{e}(t')\vert\right) 
        &\leq L\int_0^t \mathbf{E}\vert \mathfrak{e}(s)\vert ds +   
        C\left(\frac{\gamma}{\delta_\gamma^2}T + \frac{\gamma}{\delta_\gamma^3}T+\frac{1}{\delta_\gamma}+\frac{1}{\delta_\gamma^2}\right)\ve^{1/2} \\
        &\quad + \ve^{1/4}T^{1/2}\mathbf{E}\left(\sup_{0\leq t' \leq t}\vert M(t') \vert\right) \\
        &\leq L\int_0^t \mathbf{E}\vert \mathfrak{e}(s)\vert ds +   
        C\left(\frac{\gamma}{\delta_\gamma^2}T + \frac{\gamma}{\delta_\gamma^3}T+\frac{1}{\delta_\gamma}+\frac{1}{\delta_\gamma^2}\right)\ve^{1/2} \\
        & \quad + C\ve^{1/4}T^{1/2} \mathbf{E}\vert \left< M\right>_t\vert^{1/2} \\
        &\leq L\int_0^t \mathbf{E}\sup_{0\leq \tau \leq s}\vert \mathfrak{e}(\tau)\vert ds
        + C\left(\frac{\gamma}{\delta_\gamma^2}T + \frac{\gamma}{\delta_\gamma^3}T+\frac{1}{\delta_\gamma}+\frac{1}{\delta_\gamma^2}\right)\ve^{1/2} \\
        &\quad + C\left(\frac{\gamma}{\delta_\gamma^2} + 1\right)\ve^{1/4}T^{1/2} {.}
    \end{align*}
    By the integral version of the Gronwall inequality, 
    we obtain
    \begin{align*}
        &\quad \mathbf{E}\left(\sup_{0\leq t \leq t_f}\vert \mathfrak{e}(t)\vert\right)\\
        &\leq C \left[\left(\frac{\gamma}{\delta_\gamma^2}T + \frac{\gamma}{\delta_\gamma^3}T+\frac{1}{\delta_\gamma}+\frac{1}{\delta_\gamma^2}\right)\ve^{1/2} + \left(\frac{\gamma}{\delta_\gamma^2} + 1\right)\ve^{1/4}T^{1/2}\right] \\
        & \leq C\left[\left(\frac{1}{\delta_\gamma}+\frac{1}{\delta_\gamma^2}\right)\ve^{1/2} + \left(\frac{\gamma}{\delta_\gamma^2}+1\right)\ve^{1/4}T^{1/2} + \frac{\gamma}{\delta_\gamma^3}\ve^{1/2}T\right]{.} 
    \end{align*}
\end{proof}

Now we move on to Theorem~\ref{thm:Averaged_Error}. 
Compared to the exact dynamics, the averaged equation
formally only involves one additional term, which can be handled by  
the variational equation.

\begin{proof}[Theorem~\ref{thm:Averaged_Error}]
    Define $\Psi(t,s, \eta, \zeta)$ to be the resolvent of the variational equation 
    \[
        \begin{split}
            \dot{\Psi}(t,s,\eta,\zeta) &= 
            \left(\begin{array}{cc}
                0 & I_d \\
                \frac{\partial \Bar{h}}{\partial r}(\mathfrak{u}^{s,t}(\eta,\zeta)) & 0
            \end{array}
            \right)
            \Psi(t,s,\eta,\zeta) {,}\\
            \Psi(s,s,\eta,\zeta) &= I_{2d}
        \end{split}
    \]
    where $\mathfrak{u}^{s,t}(\eta,\zeta)$ is the solution to the averaged 
    equation~\eqref{eqn:Averaged_ODE} 
    with starting time at $s$ and 
    initial value $r(s) = \eta$ and $p(s) = \zeta${.}
    % Note that the solution of the variational equation can be given by
    % \[
    %     \Psi(t,s,\eta,\zeta) = \mathfrak{T}\exp\left(\int_s^t \left(\begin{array}{cc}
    %             0 & I_d \\
    %             \frac{\partial \Bar{h}}{\partial r}(\mathfrak{u}^{s,t'}(\eta,\zeta)) & 0
    %         \end{array}
    %         \right)dt'\right) 
    % \]
    % where the time-ordered exponential is defined as
    % \[
    %     \mathfrak{T}\exp \left(\int_s^t B(t')dt'\right) = I + \int_s^t B(t')dt' + \sum_{n=2}^{\infty} \int_s^t\int_s^{t'_n}\cdots\int_s^{t'_2}B(t'_n)\cdots B(t'_1)dt'_1\cdots dt'_n {.}
    % \]
    By assumption 1, 2 and 3, $\frac{\partial \Bar{h}}{\partial r}$ is bounded 
    independently of $\eta,\zeta$, thus 
    $\Psi$ is bounded independently of $T$. 
    
    Then by the theorem of Alekseev and Gr\"obner~\cite[Theorem
    14.5]{HairerNorsettWanner1987}, %\JL{this is just Duhamel?}\LL{I thought the fixed linear operator version is called Duhamel?}\JL{okay. Better give a more precise citation in the book, like \cite[Theorem xxx]{HairerNorsettWanner1987}}\DA{changed}
    \[
        \begin{pmatrix}
        \rave(t)\\
        \drave(t)
        \end{pmatrix} = 
         \begin{pmatrix}
        r_{\star}(t)\\
        \dot{r}_{\star}(t)
        \end{pmatrix} + 
        T\int_0^t \Psi(t,s,\rave(s),\drave(s))
            \begin{pmatrix}
        0\\
        g(\rave(s))
        \end{pmatrix} ds {.}
    \]
We hereby obtain the desired estimate. 
\end{proof}

\section{Numerical examples}\label{sec:numer}

In this section we verify the accuracy and the order of convergence indicated in
Theorem~\ref{thm:Main_Theorem}. We also demonstrate the efficiency of
Stochastic-XLMD in terms of the reduction of the number of SCF iterations, \ie 
the number of iterations in solving Eq.~\eqref{eqn:General_Dynamics_b} with 
iterative methods. 
%\JL{seems to be the first time ``SCF iterations'' is used}\DA{I add a very short explanation in the previous sentence. }
 We demonstrate the accuracy and efficiency of the Stochastic-XLMD method for model polarizable force field calculations in Section~\ref{sec:accuracy} and~\ref{sec:efficiency}.
Although our theory is developed for interaction energy $Q$ that is quadratic with respect to $x$, numerical results indicate that the Stochastic-XLMD method is also applicable to $Q$ that has more general dependence on $x$. In Section \ref{sec:general_interaction} we provide such results for a model problem. We further demonstrate the application to a realistic polarizable water problem with long time simulation in Section \ref{sec:water}.
All the calculations for the model problems were carried out
using MATLAB on the Berkeley Research Computing program at the
University of California, Berkeley.
Each node consists of two Intel Xeon 10-core Ivy Bridge processors (20 cores
per node) and 64 GB of memory. %The calculations for the polarizable water were carried out \LL{ Update the setup}

\subsection{Accuracy}\label{sec:accuracy}
Let us consider a simple two dimensional model 
\begin{equation*}
    U(r) = r_1^2 + r_2^2 = |r|^2 {,} \quad F(r) = - \frac{\partial U}{\partial r} {,}
\end{equation*}
\begin{equation*}
    A(r) = \left(\begin{array}{cc}
        2+|r|^2 & |r|^2 \\
        |r|^2 & 1+|r|^2
    \end{array}\right) {,}
\end{equation*}
\begin{equation*}
    b(r) = (\sin(r_1+r_2), \cos(r_1-2r_2))^{\top} {.}
\end{equation*}
Initial values for the exact MD are
\begin{equation}
    r_{\star}(0) = (0.587, -0.810)^{\top}, \quad p_{\star}(0) = (-1.00,0.500)^{\top} {.}
\end{equation}
Initial values for the Stochastic-XLMD are
\[
    \ret(0) = r_{\star}(0), \quad \pet(0) = p_{\star}(0), 
    \quad \xet(0) = x_{\star}(0) + (0.500,-0.500)^{\top}, 
    \quad \yet(0) = (0,0)^{\top} {.}
\]

The Verlet scheme is used to propagate the exact MD, 
and the BAOAB scheme~\cite{LeimkuhlerMatthews2015} is used to propagate the Stochastic-XLMD. 
The time step size is fixed to be $5.00\times 10^{-6}$, 
which is small enough for all the numerical solutions generated 
in this subsection to be regarded as the exact analytic solution 
under the same parameters. 
Other than the long time simulation reported at the end of this subsection, the time interval is fixed to be $[0,5]$, and all reported errors are the averaged errors of 10 independent simulations. 

First, Theorem~\ref{thm:Main_Theorem} assumes that $\gamma$ should be
$\Or(1)$. To confirm that such choice
can yield the optimal error, we adjust $\gamma$ with respect to various choices of $\ve$
and $T$.  Figure~\ref{fig:comp_SDE_gamma} indicates that in order to
minimize the error, the optimal value of $\gamma$ is indeed a constant
and is around $0.100$ for this example.

\begin{figure}
    \centering
    \includegraphics[width=0.48\linewidth]{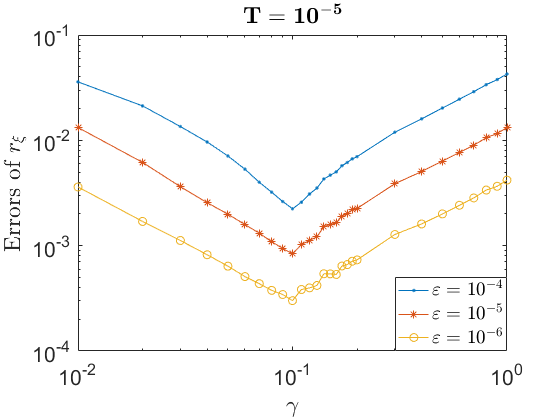}
    \includegraphics[width=0.48\linewidth]{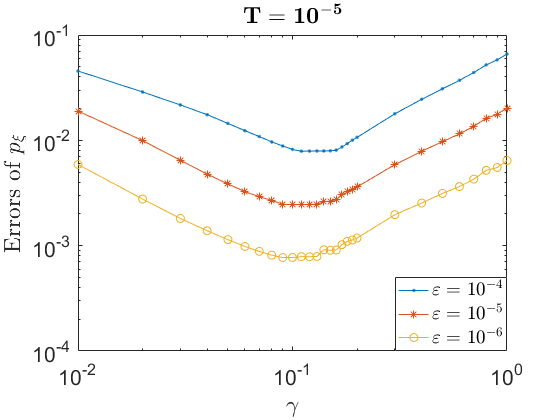}\\
    \includegraphics[width=0.48\linewidth]{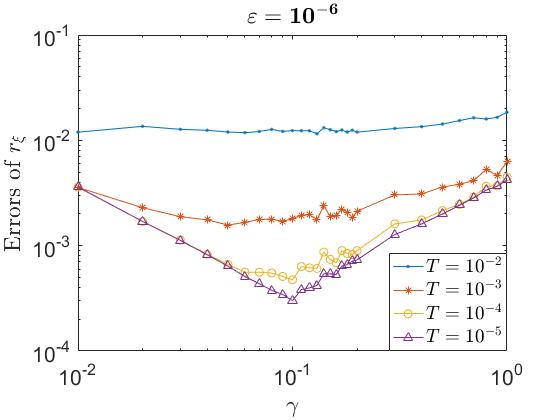}
    \includegraphics[width=0.48\linewidth]{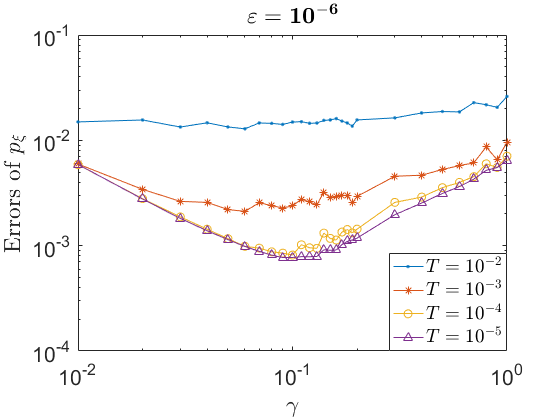}\\
    \caption{Errors of $\ret$ and $\pet$ under different choices of $\gamma$. }
    \label{fig:comp_SDE_gamma}
\end{figure}

Now we fix $\gamma = 0.100$ and $T = 10^{-5}$ 
and study the dependence on $\ve$. 
Figure~\ref{fig:comp_SDE_ep} shows that 
under such choice of $\gamma$ and $T$, 
the errors of $r$ and $p$ decrease as $\ve$ becomes smaller. 
The order of convergence, estimated using data points with
$\ve \leq 10^{-4}$, is $0.402$ for $r$ and $0.509$ for $p$. 
Furthermore, there is no essential difference among different choices of 
$T = 10^{-4}, 10^{-5}, 10^{-6}$. 
This is because $T$ is sufficiently small so that the error is dominated by 
the averaging error shown in Theorem~\ref{thm:SDE_Error}. Also, 
since $T$ is very small, the $\Or(\ve^{1/4}T^{1/2})$ term almost vanishes and 
we can only observe the half order convergence with respect to
$\ve$. 

Then we fix $\gamma = 0.100$ and study the dependence on $T$ 
with $\ve = 10^{-4},10^{-5},10^{-6}$.
Figure~\ref{fig:comp_SDE_Ts} shows that when $T$ decreases, 
the errors of $r$ and $p$ decrease accordingly, until limited by the
systematic error due to $\ve$. 
The numerical order of convergence, estimated using the first five points 
with $\ve = 10^{-6}$, is $0.964$ for $r$ and $0.933$ for $p$. 
In this case, $\ve$ is small enough and we can only observe the
$\ve$-independent part of the contribution of the error as described in Theorem~\ref{thm:Averaged_Error}. 

Our analysis indicates that the optimal strategy for choosing $\ve$ and
$T$ is that $T \sim \sqrt{\ve}$. To confirm this,
Figure~\ref{fig:comp_SDE_ep_Ts} shows the errors with $\gamma = 0.100$ and
$T = \sqrt{\ve}$. Under such scaling, both $r$ and $p$ converges as
$\ve\to 0$.  
The order of convergence, 
estimated by data points with $\ve \leq 2.00\times 10^{-4}$, 
is $0.505$ for $r$ and $0.506$ for $p$. 
This yields excellent agreement with Theorem~\ref{thm:Main_Theorem}. 

All the numerical convergence orders are collected in Table~\ref{tab:Conv_Order}.

To conclude this example, we perform a long time simulation up to $T_f = 100$ and observe how the errors of Stochastic-XLMD accumulate in energy, which is computed as 
$$E_{\xi}(t) = \frac{1}{2}|p_{\xi}(t)|^2 + U(r_{\xi}(t)) + \frac{1}{2}x_{\xi}(t)^{\top}A(r_{\xi}(t))x_{\xi}(t) - b(r_{\xi}(t))^{\top}x_{\xi}(t).$$
We fix the parameter $\gamma = 0.1$, and choose $T = \sqrt{\ve}$ with different choices of $\ve$. Unlike previous short time simulations, we only perform a single long time simulation for each choice of parameter. 
Figure~\ref{fig:SDE_long_time} shows the results with $\ve = 5\times 10^{-5}$ and $\ve = 10^{-5}$, together with the exact energy of the system (around 1.537). 
We observe that, although the initial condition is artificially perturbed, resulting in the initial energy to be around 1.912, stochastic-XLMD can correct the energy within a few time steps. 
Specifically, in this example, the error of energy is corrected to be very close to the exact energy within $t = 0.3$. 
As we proved, smaller $\ve$ results in smaller energy drift. 
Furthermore, numerically the long time drift of the energy seems mild and grows linearly with respect to time.

\begin{figure}
    \centering
    \includegraphics[width=0.6\linewidth]{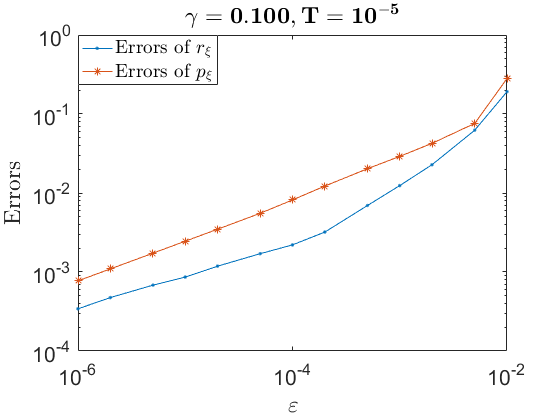}\\
    \caption{Errors of $\ret$ and $\pet$ under different choices of
    $\ve$}
    \label{fig:comp_SDE_ep}
\end{figure}

\begin{figure}
    \centering
    \includegraphics[width=0.48\linewidth]{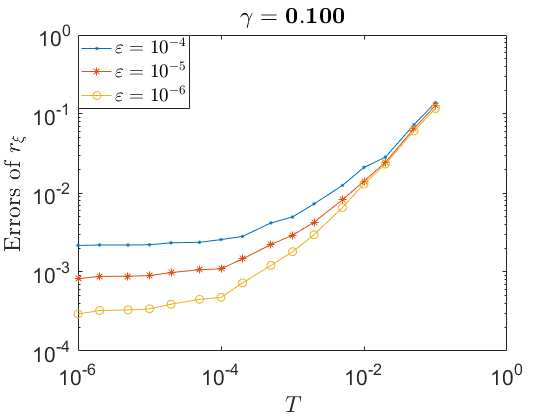}
    \includegraphics[width=0.48\linewidth]{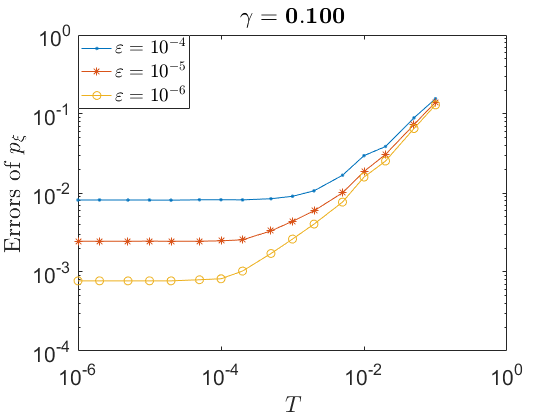}\\
    \caption{Errors of $\ret$ and $\pet$ under different choices of $T$. }
    \label{fig:comp_SDE_Ts}
\end{figure}

\begin{figure}
    \centering
    \includegraphics[width=0.6\linewidth]{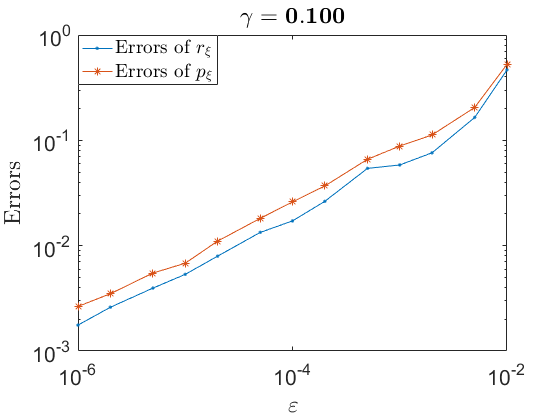}\\
    \caption{Errors of $\ret$ and $\pet$ under different choices of $\ve$
    and $T$ with $\ve = \sqrt{T}$.}
    \label{fig:comp_SDE_ep_Ts}
\end{figure}

\begin{figure}
    \centering
    \includegraphics[width=0.48\linewidth]{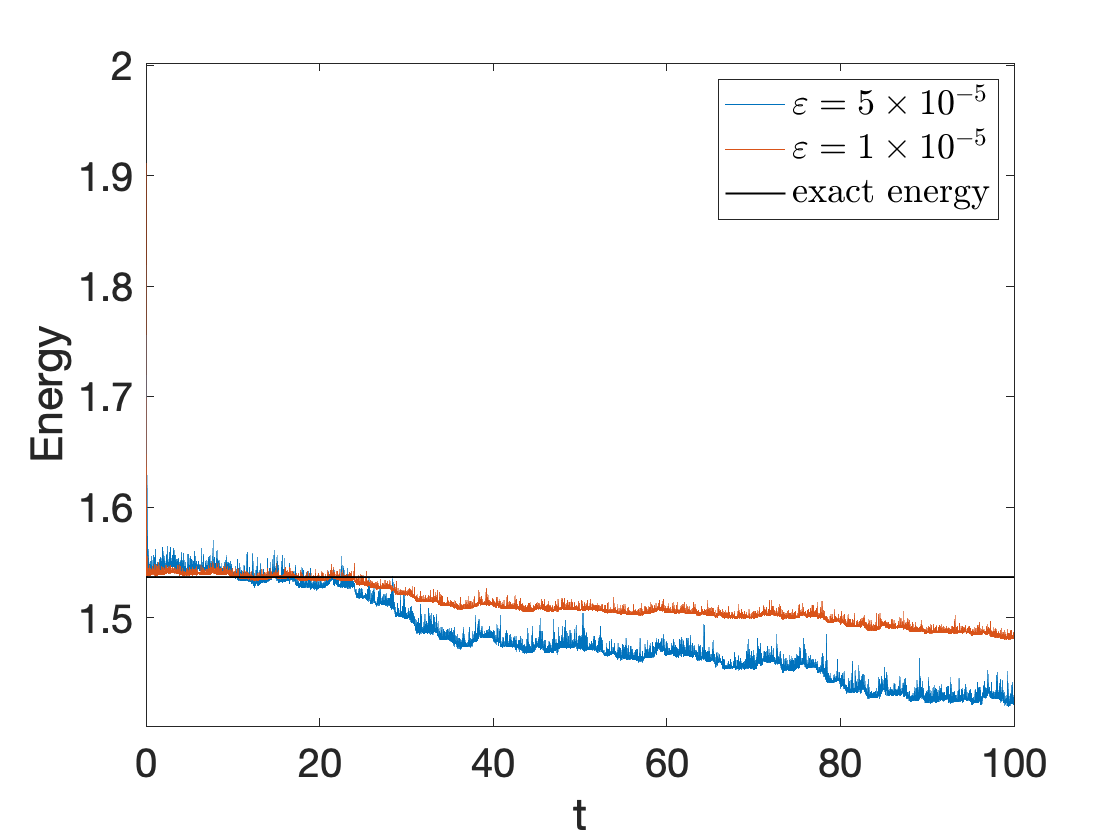}
    \includegraphics[width=0.48\linewidth]{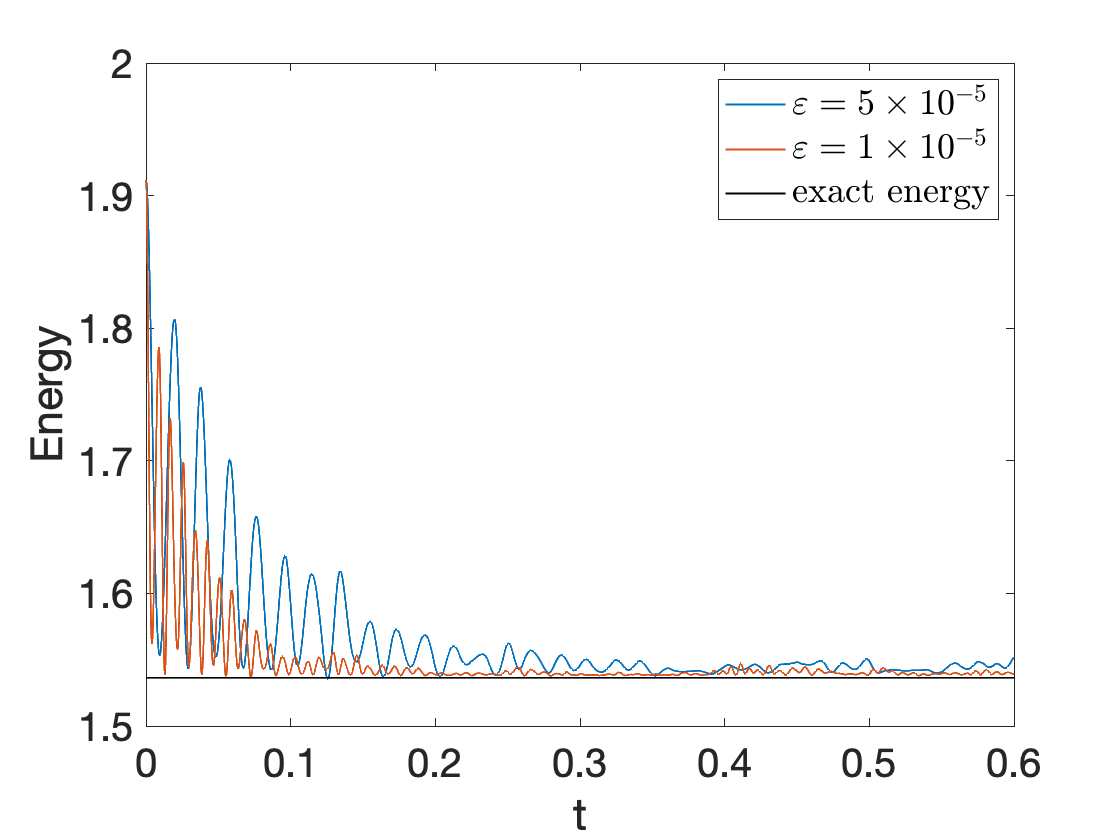}
    \caption{Energy of the long time simulation (left) and a zoom-in view at the beginning (right) under different choices of $\varepsilon$, $T = \sqrt{\varepsilon}$ and $\gamma = 0.1$.}
    \label{fig:SDE_long_time}
\end{figure}

\begin{table}
    \centering
    \begin{tabular}{cc|cc}
        fixed parameter & variable & order for $\ret$ & order for $\pet$ \\\hline
       $\gamma = 0.100$, $T = 10^{-5}$  & $\ve$    & 0.402 & 0.509 \\
       $\gamma = 0.100$, $\ve = 10^{-6}$    & $T$  & 0.964 & 0.933 \\ 
       $\gamma = 0.100$  & $\ve$, $T = \sqrt{\ve}$ & 0.505 & 0.506 
    \end{tabular}
    \caption{Numerical convergence orders for $\ret$ and $\pet$.}
    \label{tab:Conv_Order}
\end{table}

\subsection{Efficiency}\label{sec:efficiency}

After establishing the accuracy of Stochastic-XLMD method, we demonstrate
that with proper choice of parameters, Stochastic-XLMD indeed improves the
efficiency by reducing the number of iterations for solving the
nonlinear system~\eqref{eqn:General_Dynamics_b}. This is the case for
the polarizable force field model as proved in
Theorem~\ref{thm:Main_Theorem}. 
%Reduction of SCF Iterations
%
%In physical practice, 
%0-SCF schemes are designed to use much less 
%SCF iterations to capture relatively long term dynamics, allowing 
%fast oscillation at small scale. 
%We would like to show this performance by the following simple examples. 

%\subsubsection{Polarizable force field model}
%\DA{This part has been replaced by a new example with higher frequency and 
%smaller amplitude in external potential, which I think maybe better for 
%Stochastic-XLMD. }
 
Let $r\in \R^3$ and $x \in \R^{20}$. Consider 
$F = -\partial U/\partial r$ with 
\[
    U = \frac{1}{4}|r|^4 + \frac{1}{100}\cos(400(r_1+r_2+r_3)){.}
\]
For the polarizable force field model, the non-zero entries in $A$ are
given by $A_{k,k} = 2+|r|^2$, $A_{k,k+1} = A_{k+1,k} = -1$, $A_{k,k+2} =
A_{k+2,k} = (1 - |r|^2)/2$.  $b_k = \sin(kr_2/10 + (1-k/20)r_2 + r_3), k = 1,\cdots, 20$. %\JL{please indicate range of the $k$ variable}
The choice of parameters are motivated from practical polarizable force
field calculations, where the force $F$ is strong and dominates the
dynamics at short time scale, while the interaction energy affects the
dynamics at long time scale. 
The time interval is fixed to be $[0,5]$. %\JL{is this long or short time scale in light of the previous sentence?}
%\DA{this is the long time scale in light of the previous sentence, because here the dominant part of $F$ is a highly oscillatory term with time scale $\sim 2\pi/400$. }
Initial values are $r(0) = (0,0.500,1.00)^{\top}$, $p(0) = (1.00,0.500,-1.00)^{\top}$.

We compare numerical performance of MD (directly propagating MD~\eqref{eqn:General_Dynamics}) and Stochastic-XLMD.
For MD, we use the Verlet scheme to propagate the dynamics, and use
the conjugate gradient method (CG) to solve the SCF
iterations (i.e., solving the linear system).  The reference solution is
obtained with MD with a very small time step size $2.50\times 10^{-6}$,
and the SCF tolerance (measured in terms of the residue 
$|b-Ax|$) is set to $10^{-10}$.  
For Stochastic-XLMD, the BAOAB scheme is used for time propagation, and
the time step size $1/2500$.  Other parameters are chosen to be $\ve =
5.00\times 10^{-7}$, $T = \sqrt{\ve}/1000$, $\gamma = 0.500$. 
In order to demonstrate the efficiency of Stochastic-XLMD, we perform 
MD simulation with the same time step size $1/2500$. 
The stopping criteria is set to be $10^{-6}$. We remark that such choice of
tolerance is at the threshold, in the sense that the error of $p$ indeed
increases if we set the tolerance to be larger. %\JL{this sentence is not clear; do we mean that it is at the threshold?}\DA{yes. I modify the sentence accordingly. }
Such parameters are chosen such that all the dynamics are almost 
indistinguishable with the reference solution till $t = 3$ and 
remain reasonably accurate within the whole time interval. 
See Figure~\ref{fig:comp_MD_SDE} for a comparison of $r$ and $p$ obtained by
different methods. 
%\LL{Indeed XLMD requires more gradient evaluation. What if the time step of two schemes are chosen to be the same? Another question is that how does Stochastic-XLMD compare with standard XLMD? Also I am not sure the accuracy of Stochastic-XLMD is good enough: judged from the trajectory of $r_1$ there is large deviation between $3<t<4$.}
%\DA{Here the time step for Stochastic-XLMD has reached the stability limit, and the time step for MD is chosen so that errors in two methods are comparable. The time step for MD is larger because the convergence for MD with Verlet is second order, while Stochastic-XLMD is first order(error $\sim \Or(\sqrt{\ve})$, and the limit of time step to ensure stability is $h \sim \sqrt{\ve}$). Here the stability limit for MD is still large. I am testing other examples, increasing the frequency of external potential, together with comparing with standard XLMD. }

\begin{table}
    \centering
    \begin{tabular}{c|cc|cc}
        Method & Errors of $r$ & Errors of $p$ & Number of $Ax$ & Number of $(\partial_{r_k}A)x$ \\\hline
      MD            & 0.0507 & 0.228 & 100392 &  37503  \\
      Stochastic-XLMD & 0.0401 & 0.295 & 12518  &  37503 \\
    \end{tabular}
    \caption{Numerical errors and computational costs of 
    MD and Stochastic-XLMD applied to the polarizable force field model. }
    \label{tab:Comp_MD_SDE}
\end{table}

\begin{figure}
    \centering
    \includegraphics[width=0.48\linewidth]{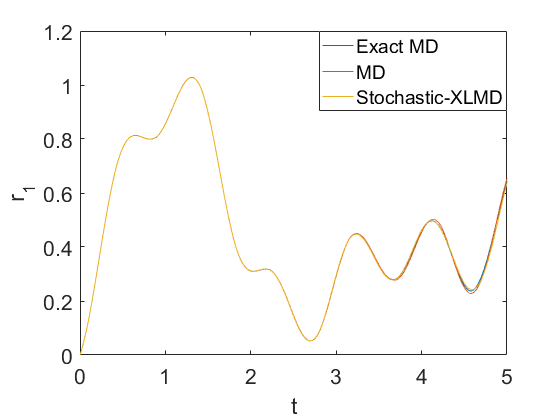}
    \includegraphics[width=0.48\linewidth]{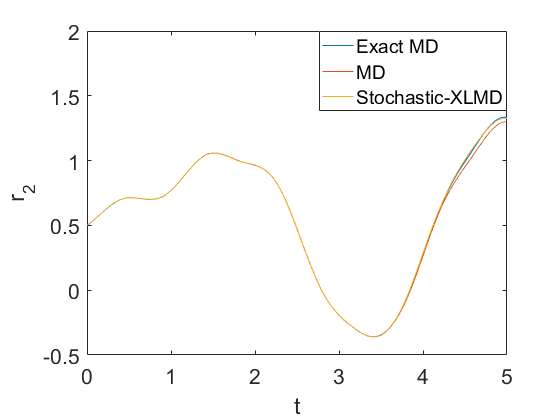}\\
    \includegraphics[width=0.48\linewidth]{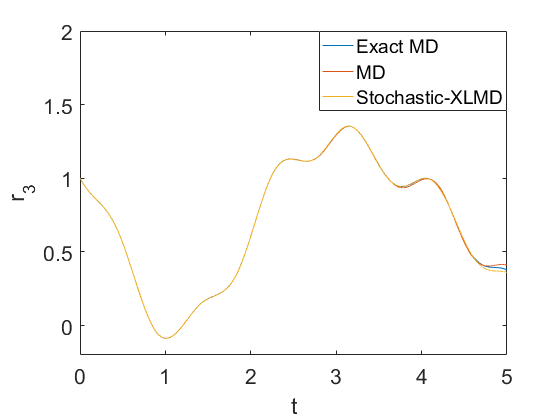}
    \includegraphics[width=0.48\linewidth]{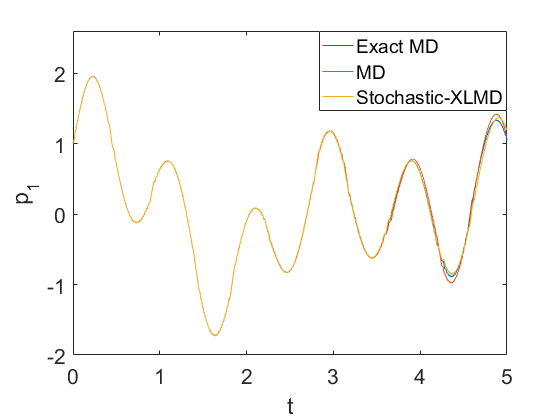}\\
    \includegraphics[width=0.48\linewidth]{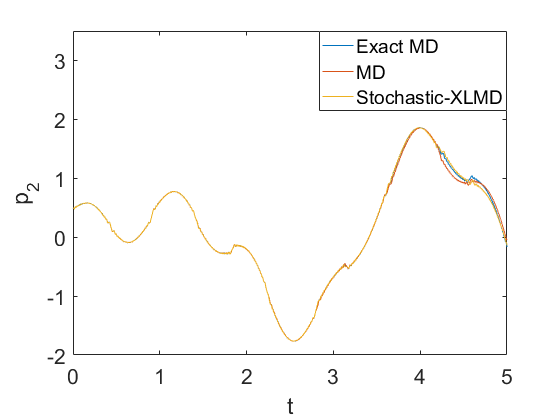}
    \includegraphics[width=0.48\linewidth]{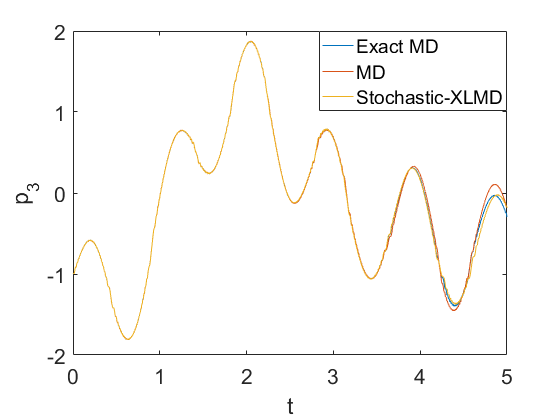}\\
    \caption{Comparison of $r$ and $p$ obtained by MD and 
    Stochastic-XLMD applied to the polarizable force field model. }
    \label{fig:comp_MD_SDE}
\end{figure}

Table~\ref{tab:Comp_MD_SDE} compares numerical errors and computational 
costs of MD and Stochastic-XLMD. Here the error in Stochastic-XLMD reported 
is computed by taking average of 10 independent simulations. 
The computational cost is measured by the number of matrix-vector multiplications.  
In each time step, the number 
of $Ax$ is equal to the number of SCF iterations plus one. 
We find that Stochastic-XLMD achieves similar accuracy compared to MD,
but reduces the number of SCF iterations by 87.5\%. After taking into
account the matrix-vector multiplication operations due to
$(\partial_{r_k}A) x$ for computing the force, Stochastic-XLMD still
reduces the total matrix-vector multiplications by  63.7\%.

%\FloatBarrier

\subsection{General form of interaction energy}\label{sec:general_interaction}

Numerical results indicate that
the same behavior can also be observed for more general interaction
energy that is non-quadratic with respect to $x$ as well. In both cases,
the interaction energy is nonlinear with respect to $r$.

Next we test the effectiveness and efficiency of Stochastic-XLMD 
applied to a system with interaction energy $Q$ that is non-quadratic
with respect to $x$.
More specifically, we set 
\[
    Q = \frac{1}{2}x^{\top}A(r)x - x^{\top}b(r) + 0.150\left(\vert x\vert^2+\frac{1}{2}\sum_{k=1}^{20}\sin(2x_k)\right) {.}
\]
Such choice of $Q$ will ensure that the Hessian matrix with respect to 
$x$ is uniformly positive definite, which means that the system of 
nonlinear equations 
\[
    0 = -\frac{\partial Q}{\partial x} 
\]
has a unique solution and the dynamics is well-defined. 

We use Anderson mixing without preconditioning~\cite{Anderson1965} to solve the system 
of nonlinear equations, and all other numerical treatments remain 
to be the same. 
In Anderson mixing, the SCF tolerance is chosen to be $10^{-6}$. 
Such choice is again relatively tight, and further increase of the tolerance 
will increase the numerical errors in both $r$ and $p$. 
The mixing parameter $\alpha$ is set to be $0.100$ to ensure convergence, and
the mixing dimension is $5$. 
The reference solution is obtained with very small time step size $2.50\times 10^{-6}$.  
In the MD simulation, the time step size is chosen to be
$1/2000$, 
while the time step size in Stochastic-XLMD is $1/2500$. 
Other parameters in Stochastic-XLMD are 
$\ve = 2.50\times 10^{-7}$, $T = \sqrt{\ve}/10000$, $\gamma = 0.100$. 
Again, such parameters are chosen for all the dynamics to be almost 
indistinguishable with the reference solution till $t = 3.5$ and 
remain reasonably accurate within the whole time interval. 
See Figure~\ref{fig:comp_MD_SDE_nonlinear} for a comparison of $r$ and $p$
obtained by different methods.
%\JL{I am curious if we run the trajectory using S-XLMD for even longer time period, does it remain stable? Just wondering about long time stability and error accumulation} \DA{I tried with toy model with $t_f = 100$. S-XLMD remains stable. The error is oscillating and seems like linear growth. }

\begin{figure}
    \centering
    \includegraphics[width=0.4\linewidth]{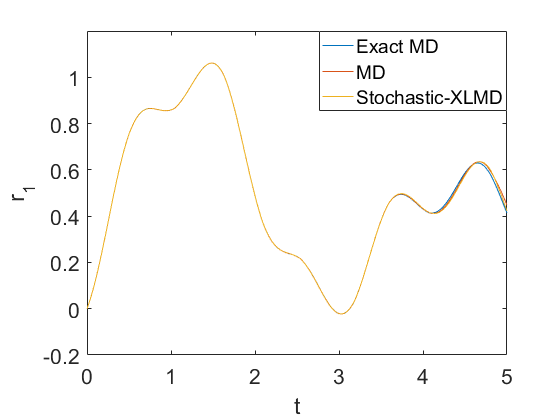}
    \includegraphics[width=0.4\linewidth]{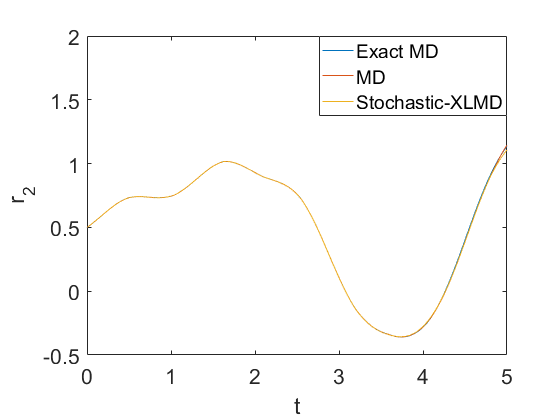}\\
    \includegraphics[width=0.4\linewidth]{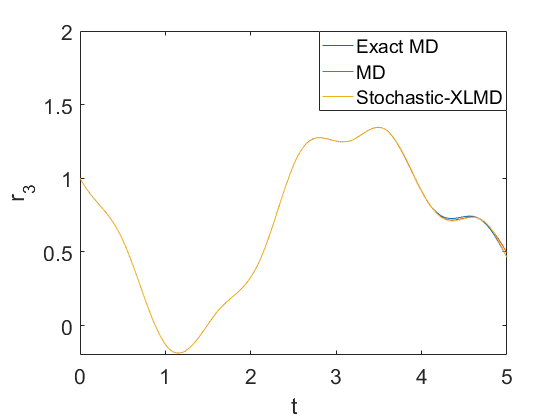}
    \includegraphics[width=0.4\linewidth]{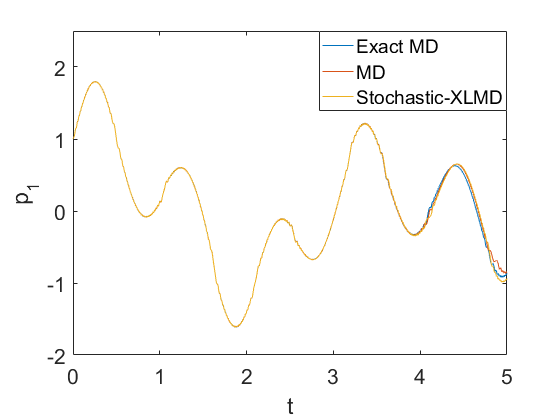}\\
    \includegraphics[width=0.4\linewidth]{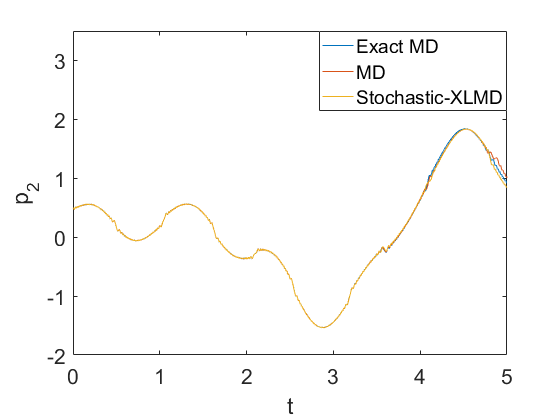}
    \includegraphics[width=0.4\linewidth]{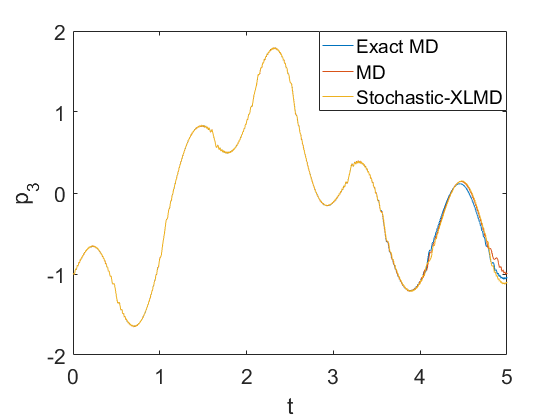}\\
    \caption{Comparison of $r$ and $p$ obtained by MD and 
    Stochastic-XLMD applied to the nonlinear example. }
    \label{fig:comp_MD_SDE_nonlinear}
\end{figure}

\begin{table}
    \centering
    \begin{tabular}{c|cc|c}
        Method & Errors in $r$ & Errors in $p$ & Number of nonlinear evaluations  \\\hline
      MD            & 0.0459 & 0.318 & 128763 \\
      Stochastic-XLMD & 0.0540 & 0.301 & 12601  \\
    \end{tabular}
    \caption{Numerical errors and computational costs of 
    MD and Stochastic-XLMD applied to the nonlinear example. }
    \label{tab:Comp_MD_SDE_nonlinear}
\end{table}

Table~\ref{tab:Comp_MD_SDE_nonlinear} compares numerical errors and
computational costs of MD and Stochastic-XLMD. 
Here the error in Stochastic-XLMD reported 
is computed by taking average of $10$ independent simulations. 
The computation cost 
is measured by the number of the number of nonlinear evaluations, 
in particular, the number of evaluating $\partial Q/ \partial x$. 
In each time step, this number is equal to the number of SCF iterations
plus one.  Similarly with the polarizable force field model, numerical errors of 
MD and Stochastic-XLMD are comparable, while 90.2\% 
of nonlinear evaluations are reduced by using Stochastic-XLMD. 

\subsection{Polarizable model for water}\label{sec:water}

\begin{figure}
    \centering
    \includegraphics[width=1\linewidth]{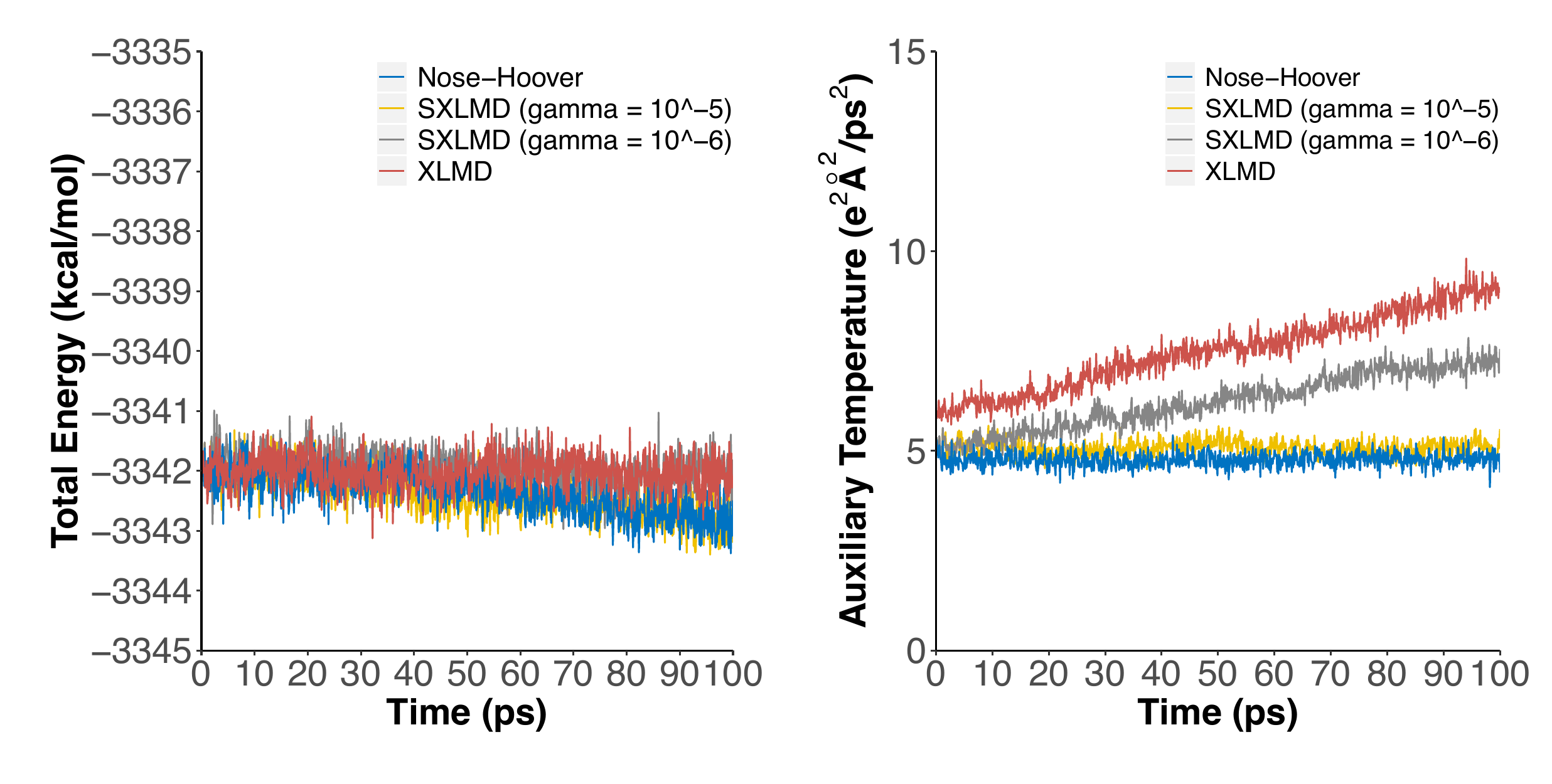}
    \caption{Comparison of total energy drift rates and auxiliary temperatures
    obtained  using XLMD, Nose Hoover iEL/0-SCF, and the Stochastic-XLMD method applied to the polarizable water AMOEBA14~\cite{Laury2015} potential energy model. (a) Energy drift rates (kcal/mol/ps) were found to decrease in order from Nose-Hoover iEL/0-SCF  (blue, $-9.032\times 10^{-3}$), Stochastic-XLMD with $\gamma=10^{-5}$ (yellow, $-8.431\times 10^{-3}$), Stochastic-XLMD with $\gamma=10^{-6}$ (grey, $-3.139\times 10^{-3}$), and XLMD (red, $-7.518\times 10^{-4}$).  For comparison, the total energy drift rate with a standard pre-conditioned conjugate-gradient self-consistent field method\cite{Wang2005} is  $-8.326\times 10^{-4}$ kcal/mol/ps using a $10^{-6}$ RMS Debye dipole convergence criteria. (b) the auxiliary temperature over time shows that  there is a kinetic buildup of error that is not dissipated in the XLMD approach, whereas the kinetic energy is well-dissipated by the thermostatting methods. The MD simulation is comprised of 512 water molecules simulated for 100 ps in the NVE Ensemble with a 0.5 fs time-step, and energy and temperature reported at a 0.1 ps output rate. For the Nose-Hoover iEL/0-SCF and Stochastic-XLMD results, $\gamma_{\text{aux}}=0.9$. }
    \label{fig:long_time2}
\end{figure}

We have also applied the Stochastic-XLMD approach to a more realistic atomic polarizable model for 512 water molecules simulated with the AMOEBA polarizable force field~\cite{Laury2015}. Figure~\ref {fig:long_time2}a provides a comparison of energy conservation in the NVE ensemble (a vital quantity for correct Hamiltonian dynamics) between the original iEL/0-SCF method~\cite{AlbaughNiklassonHead-Gordon2017} and Stochastic-XLMD and XLMD simulations. The difference in methods applied to this real world polarizable  system resides in the treatment of the auxiliary thermostats, and thus we also report the kinetic energy proxy for the latent variables in Figure~\ref {fig:long_time2}b.

For the original iEL/0-SCF approach, the temperature of the auxiliary degrees of freedom are controlled with a 4th order Nose-Hoover thermostat~\cite{AlbaughDemerdashHead-Gordon2015,AlbaughNiklassonHead-Gordon2017}, and requires the determination of an optimal value of $\gamma_{\text{aux}}$=$dt^{2}/(2\varepsilon$) for best energy conservation~\cite{AlbaughDemerdashHead-Gordon2015,AlbaughNiklassonHead-Gordon2017}, where $dt$ is the time step size. For the Stochastic XLMD method we have determined optimal $\gamma$ values of $10^{-5}$ to $10^{-6}$ for $\gamma_{\text{aux}}=0.9$ to generate acceptable energy drift on par with the original iEL/0-SCF. 
%In unpublished work, An and Lin have shown that the XLMD method, i.e. using no latent variable thermostats, yields by far the best energy conservation compared to any thermostatted method. 
In fact the energy drift rate of the XLMD approach is comparable to that of a standard self-consistent field iterative procedure~\cite{Wang2005} with reasonably tight convergence, as is confirmed in Figure~\ref {fig:long_time2}a. This would suggest that the thermostatted methods offer no significant advantage to XLMD!

However, trajectories of the auxiliary temperature over time shows that while XLMD does conserve energy better than Stochastic-XLMD or Nose-Hoover on the short timescale, there is a kinetic buildup of error that is not dissipated in the XLMD approach (Figure~\ref {fig:long_time2}b). This corruption of the auxiliary dynamics will ultimately feed back into the real degrees of freedom, creating ``resonances'' that will result in long-term instability of the XLMD algorithm.   By contrast, Stochastic-XLMD and Nose-Hoover iEL/0-SCF methods control the kinetic energy buildup better than XLMD, as expected. For a choice of  $\gamma$=$10^{-5}$, the Stochastic-XLMD method is a good compromise between energy conservation and long term stability; furthermore Stochastic-XLMD is an excellent alternative to Nose-Hoover thermostats because of its lighter weight overhead compared to thermostatted chains.
 
\section{Conclusion}\label{sec:conclusion}

In this work, we consider a stochastic-extended Lagrangian molecular
dynamics method, by introducing numerical fluctuation and dissipation 
%\JL{see comment in the introduction about ``dissipation''} 
through a Langevin type thermostat.
For a simple polarizable force field model, with a suitable choice of the
Lagrangian, we yield the Stochastic-XLMD method which generalizes the recently
proposed iEL/0-SCF method~\cite{Niklasson2012,AlbaughNiklassonHead-Gordon2017,AlbaughHead-Gordon2017}.
%Such dissipation can improve the accuracy of the dynamics affected by
%the error of the initial induced dipole field in iEL/0-SCF. 
We prove that the Stochastic-XLMD method converges to accurate
dynamics, and the convergence rate is sharp with respect to the singular perturbation parameter
$\ve$ and the numerical temperature $T$. % The convergence is independent of the initial value of the latent variable.
We also analyze 
the impact of the damping factor in the
Langevin dynamics and identify the optimal choice. 
%Such result provides a general theoretical understanding and suggestion
%on parameter choice for Langevin-type thermostat. 
While our analysis is done for a simple polarizable force field model where the
interaction energy is quadratic with respect to the latent degrees of
freedom, we have shown that our results can be generalized 
to accommodate more general interaction energy forms such as the atomistic polarizable model AMOEBA for liquid water.\cite{Laury2015} 
%THG Isn't next sentence repetitive to previous paragraph? 
%DA I think the next sentence is from the numerical perspective, while the previous discussion focuses on theoretical analysis? They did look similar and repetitive. I slightly modify the expression. 
Interesting future directions include theoretical understanding of the
convergence of the Stochastic-XLMD scheme for other models such as the
Kohn-Sham density functional theory or for reactive force fields \cite{ReaxFF}, and the convergence of the original
iEL/0-SCF scheme in the absence of noise. 
%Theoretical understanding and more detailed and practical numerical 
%simulation of Stochastic-XLMD applied to non-quadratic 
%interaction case such as KS-DFT are our future work. 

\section*{Acknowledgments} 

This work was partially supported by the National Science Foundation
under grant DMS-1652330 (D.A. and L.L.) and DMS-1454939 (J.L.), by the
Department of Energy under grant DE-SC0017867 (L.L.) and the U.S.
Department of Energy, Office of Science, Office of Advanced Scientific
Computing Research, Scientific Discovery through Advanced Computing
(SciDAC) program (T.H.-G. and L.L.). S.Y.C also thanks the
Berkeley-France Fund for support of this work. We thank Berkeley Research
Computing (BRC) for computational resources. We thank Christian
Lubich, Anders Niklasson and Chao Yang for helpful discussions.

\appendix

\section{Proof of Proposition~\ref{prop:Poisson}}\label{append:proof_Prop5}

(a) Since $A$ is a positive definite matrix, there exist 
$\lambda_1 \geq \cdots \geq \lambda_d > 0$ and an orthonormal
basis $\{v_k\}_{k=1}^{d}$ of $\RR^{d}$ which satisfy 
\[
    Av_{k} = \lambda_k v_{k}, \quad k = 1, \cdots, d {.}
\]
Define 
\[
    \mathfrak{U} = \left(\mathfrak{U}_1, \cdots, \mathfrak{U}_d\right) \in \RR^{2d\times 2d}
\]
where
\[
    \mathfrak{U}_k = \frac{1}{\sqrt{2}}\left(\begin{array}{cc}
        v_k & v_k \\
        v_k & -v_k
    \end{array}\right) 
    \in \RR^{2d\times 2} {.}
\]
It is easy to check $\mathfrak{U}$ is an orthogonal matrix. Define
\[
    \mathfrak{U}^{\top} \mathfrak{B} \mathfrak{U} =: \mathfrak{J} 
    = \left(\begin{array}{ccc}
        J_{11} & \cdots & J_{1d} \\
        \vdots &  & \vdots \\
        J_{d1} & \cdots & J_{dd} 
    \end{array}\right)\in \RR^{2d\times 2d}
\]
with $J_{kl}\in \RR^{2\times 2}$ given by 
\begin{align*}
    J_{kl} = \mathfrak{U}^{\top}_{k} \mathfrak{B} \mathfrak{U}_{l} &= 
    \frac{1}{2}\left(\begin{array}{cc}
        v_k^{\top} & v_k^{\top} \\
        v_k^{\top} & -v_k^{\top}
    \end{array}\right)\left(\begin{array}{cc}
            0 & -I_d \\
            A & \gamma I_d
        \end{array}\right)
        \left(\begin{array}{cc}
        v_l & v_l \\
        v_l & -v_l
    \end{array}\right) \\ 
    &= \frac{1}{2}\left(\begin{array}{cc}
        v_{k}^{\top}Av_l+(\gamma-1)v_k^{\top}v_l & v_k^{\top}Av_l + (1-\gamma)v_k^{\top}v_l \\
        -v_k^{\top}Av_l+(-1-\gamma)v_k^{\top}v_l & -v_k^{\top}Av_l + (1+\gamma)v_k^{\top}v_l
    \end{array}\right)\\
    &= \frac{1}{2}\left(\begin{array}{cc}
        (\lambda_l+\gamma-1)v_k^{\top}v_l &  (\lambda_l-\gamma+1)v_k^{\top}v_l \\
        (-\lambda_l-\gamma-1)v_k^{\top}v_l &  (-\lambda_l+\gamma+1)v_k^{\top}v_l
    \end{array}\right){.}
\end{align*}
Note that $v_k^{\top}v_k = 1$ and $v_k^{\top}v_l = 0$ if $k \neq l$. 
Therefore $J_{kl} = 0$ if $k \neq l$ and 
\[
    \mathfrak{U}^{\top} \mathfrak{B} \mathfrak{U} = \mathfrak{J} 
    = \left(\begin{array}{ccc}
        J_{11} & & \\
         & \ddots & \\
         & & J_{dd}
    \end{array}\right)
\]
with
\[
   J_{kk} =  \frac{1}{2}\left(\begin{array}{cc}
        \lambda_k+\gamma-1 &  \lambda_k-\gamma+1 \\
        -\lambda_k-\gamma-1 &  -\lambda_k+\gamma+1
    \end{array}\right) {.}
\]
Then we have 
\begin{equation}\label{eqn:exp-Bt}
    \|e^{-\mathfrak{B}t}\|_2 = \|\mathfrak{U}^{\top}e^{-\mathfrak{J}t}\mathfrak{U}\|_2 = \|e^{-\mathfrak{J}t}\|_2 = \max_{1\leq k \leq d} \|\exp(-J_{kk}t)\|_2{.}
\end{equation}
Hence it is sufficient to find an upper bound for each $\|\exp(-J_{kk}t)\|_2$. 

For notational simplicity, we will drop the subscript for $J_{kk}$ and
$\lambda_k$, as the argument is identical for each $k$.  We have
\begin{equation}\label{eqn:cayley-hamilton}
    J^2 - \gamma J + \lambda I_2 = 0{,}
\end{equation}
which can be obtained by noticing that $x^2-\gamma x + \lambda$ is the
characteristic polynomial of $J$ and applying Cayley-Hamilton Theorem.
From Eq.~\eqref{eqn:cayley-hamilton} we have
\begin{equation}
    J^{n+2} = \gamma J^{n+1} - \lambda J^n, \quad \forall n{.}
\end{equation}
We now compute $\exp(-Jt)$ explicitly using the above recursion
relation.  Define the roots of the characteristic polynomial to be
\[
    \mu_{\pm} = \frac{\gamma \pm \sqrt{\gamma^2-4\lambda}}{2}{.}
\]
Note that $\mu_{\pm}$ can be complex if $\gamma^2 < 4\lambda$. We have
\begin{align*}
    J^{n+2} - \mu_{+}J^{n+1} &= \mu_{-}(J^{n+1} - \mu_{+}J^{n}), \\
    J^{n+2} - \mu_{-}J^{n+1} &= \mu_{+}(J^{n+1} - \mu_{-}J^{n}){,}
\end{align*}
then
\begin{align*}
    J^{n+1} - \mu_{+}J^{n} &= \mu_{-}^{n}(J - \mu_{+}I), \\
    J^{n+1} - \mu_{-}J^{n} &= \mu_{+}^{n}(J - \mu_{-}I){.}
\end{align*}
If $\gamma^2-4\lambda \neq 0$, then $\mu_{+}\neq \mu_{-}$ and we have 
\[
    J^{n} = \frac{1}{\mu_{+}-\mu_{-}}\left[\mu_{+}^{n}(J-\mu_{-}I) - \mu_{-}^{n}(J-\mu_{+}I)\right] {.}
\]
Then
\begin{equation}\label{eqn:exp-Jt}
    \begin{split}
        e^{-Jt} &= \sum_{n=0}^{\infty} \frac{1}{n!}(-1)^{n}t^{n}J^n \\
        &= \frac{J-\mu_{-}I}{\mu_{+}-\mu_{-}}\sum_{n=0}^{\infty} \frac{1}{n!}(-\mu_{+}t)^n - \frac{J-\mu_{+}I}{\mu_{+}-\mu_{-}}\sum_{n=0}^{\infty} \frac{1}{n!}(-\mu_{-}t)^n \\
        &= \frac{J-\mu_{-}I}{\mu_{+}-\mu_{-}}e^{-\mu_{+}t} - \frac{J-\mu_{+}I}{\mu_{+}-\mu_{-}}e^{-\mu_{-}t} \\
        &= M_{\gamma}^{(k)}e^{-\delta_\gamma^{(k)}t},
    \end{split}
\end{equation}
where
\begin{equation}
     \delta_{\gamma}^{(k)} = \begin{cases}
            \gamma/2, \quad 0 < \gamma \leq 2\sqrt{\lambda_k} {,} \\
            (\gamma - \sqrt{\gamma^2 - 4\lambda_k})/2, 
            \quad \gamma > 2\sqrt{\lambda_k} {,}
        \end{cases}
\end{equation}
and
\begin{equation}
    M_{\gamma}^{(k)} = \begin{cases}
             I\cos\left(\frac{\sqrt{4\lambda_k-\gamma^2}}{2}t\right) - \frac{2J_{kk}-\gamma I}{\sqrt{4\lambda_k - \gamma^2}}\sin\left(\frac{\sqrt{4\lambda_k-\gamma^2}}{2}t\right), \quad 0 < \gamma < 2\sqrt{\lambda_k}{,} \\
             \left(1+\frac{\gamma t}{2}\right)I + tJ_{kk}, \quad \gamma = 2\sqrt{\lambda_k}{,} \\
             I + \frac{J_{kk}-\mu_{-}I}{\sqrt{\gamma^2-4\lambda_k}}\left[\exp(-\sqrt{\gamma^2-4\lambda_k}t)-1\right], \quad \gamma > 2\sqrt{\lambda_k}{.}
        \end{cases}
\end{equation}
The case $\gamma = 2\sqrt{\lambda_k}$ can be obtained by taking the 
limit $\gamma \rightarrow 2\sqrt{\lambda_k}$ from either side. 

We now prove that there exists a constant $C>0$ independent of
$\gamma$ and $t$ such that
\begin{equation}\label{eqn:Mbound}
    \|M_{\gamma}^{(k)}e^{-\delta_\gamma^{(k)}t/2}\|_2 \leq C {.}
\end{equation}
In fact, if $\gamma > 3\sqrt{\lambda_k}$, then
$\frac{J_{kk}-\mu_{-}I}{\sqrt{\gamma^2-4\lambda_k}}$ is bounded
independently of $\gamma$, and $e^{-\sqrt{\gamma^2-4\lambda_k}t}$ is
bounded by 1. Thus $M_{\gamma}^{(k)}$ is bounded.  If
$0 < \gamma < \sqrt{\lambda_k}$, then by the fact that
$\frac{2J_{kk}-\gamma I}{\sqrt{4\lambda_k - \gamma^2}}$ is bounded
independently of $\gamma$, $M_{\gamma}^{(k)}$ is also already bounded.
Now we assume $\sqrt{\lambda_k} \leq \gamma \leq 3\sqrt{\lambda_k}$,
which means that $\gamma$, $\delta_\gamma^{(k)}$ and
$1/\delta_\gamma^{(k)}$ are all bounded so we can put all the $\gamma$
dependence in the constant $C$ and only focus on $t$-dependence.
Using the fact that $\vert \sin x/x\vert$ and
$\vert(e^{-x}-1)/x)\vert$ are both bounded by 1, and
$te^{-\delta_\gamma^{(k)}t}$ is also bounded, we can obtain the
desired estimate in~\eqref{eqn:Mbound}.

Finally, substitute Eq.~\eqref{eqn:exp-Jt} and estimate~\eqref{eqn:Mbound} into Eq.~\eqref{eqn:exp-Bt}, 
we obtain 
\[
    \|e^{-\mathfrak{B}t}\|_2 = \max_{1\leq k \leq d} \|\exp(-J_{kk}t)\|_2 \leq \max_{1\leq k \leq d} Ce^{-\delta_\gamma^{(k)}t/2} \leq Ce^{-\delta_\gamma t}{.}
\]
    
(b) According to Eq.~\eqref{eqn:Langevin_Sigmat}
\begin{align*}
 \| \mathfrak{S}_t - \mathfrak{S}_{\infty} \|_2 
 &\leq C\gamma T\int_t^{\infty} \| e^{-\mathfrak{B}s} \|_2 \| e^{-\mathfrak{B}^{\top}s} \|_2 ds \\
 &\leq C\gamma T \int_t^{\infty} e^{-2\delta_\gamma s}ds \\
 &\leq C\frac{\gamma}{\delta_\gamma}Te^{-2\delta_\gamma t} {.}
\end{align*}

\bibliographystyle{siam}
\bibliography{s0scfref}

\begin{thebibliography}{10}

\bibitem{Albaugh2016}
{\sc A.~Albaugh, H.~A. Boateng, R.~T. Bradshaw, O.~N. Demerdash, J.~Dziedzic,
  Y.~Mao, D.~T. Margul, J.~Swails, Q.~Zeng, D.~A. Case, P.~Eastman, L.~P. Wang,
  J.~W. Essex, M.~Head-Gordon, V.~S. Pande, J.~W. Ponder, Y.~Shao, C.~K.
  Skylaris, I.~T. Todorov, M.~E. Tuckerman, and T.~Head-Gordon}, {\em Advanced
  potential energy surfaces for molecular simulation}, J. Phys. Chem. B, 120
  (2016), pp.~9811--32.

\bibitem{AlbaughDemerdashHead-Gordon2015}
{\sc A.~Albaugh, O.~Demerdash, and T.~Head-Gordon}, {\em {An efficient and
  stable hybrid extended Lagrangian/self-consistent field scheme for solving
  classical mutual induction}}, J. Chem. Phys., 143 (2015), p.~174104.

\bibitem{AlbaughHead-Gordon2017}
{\sc A.~Albaugh and T.~Head-Gordon}, {\em {A New Method for Treating Drude
  Polarization in Classical Molecular Simulation}}, J. Chem. Theory Comput., 13
  (2017), pp.~5207--5216.

\bibitem{AlbaughNiklassonHead-Gordon2017}
{\sc A.~Albaugh, A.~M.N. Niklasson, and T.~Head-Gordon}, {\em {Accurate
  Classical Polarization Solution with No Self-Consistent Field Iterations}},
  J. Phys. Chem. Lett., 8 (2017), pp.~1714--1723.

\bibitem{Anderson1965}
{\sc D.~G. Anderson}, {\em {Iterative procedures for nonlinear integral
  equations}}, J. Assoc. Comput. Mach., 12 (1965), pp.~547--560.

\bibitem{BornemannSchutte1997}
{\sc F.~A. Bornemann and C.~Sch{\"u}tte}, {\em Homogenization of hamiltonian
  systems with a strong constraining potential}, Physica D, 102 (1997),
  pp.~57--77.

\bibitem{CarParrinello1985}
{\sc R.~Car and M.~Parrinello}, {\em Unified approach for molecular dynamics
  and density-functional theory}, Phys. Rev. Lett., 55 (1985), pp.~2471--2474.

\bibitem{Demerdash2014}
{\sc O.~Demerdash, E.~H. Yap, and T.~Head-Gordon}, {\em Advanced potential
  energy surfaces for condensed phase simulation}, Annu. Rev. Phys. Chem., 65
  (2014), pp.~149--74.

\bibitem{HairerLubichWanner2006}
{\sc E.~Hairer, C.~Lubich, and G.~Wanner}, {\em Geometric numerical
  integration: structure-preserving algorithms for ordinary differential
  equations}, Springer-Verlag Berlin Heidelberg, second~ed., 2006.

\bibitem{HairerNorsettWanner1987}
{\sc E.~Hairer, S.~P. N{\o}rsett, and G.~Wanner}, {\em {Solving ordinary
  differential equation I: nonstiff problems}}, vol.~8, Springer, 1987.

\bibitem{HairerPavliotis2008}
{\sc M.~Hairer and G.~A. Pavliotis}, {\em From ballistic to diffusive behavior
  in periodic potentials}, J. Stat. Phys., 131 (2008), pp.~175--202.

\bibitem{HohenbergKohn1964}
{\sc P.~Hohenberg and W.~Kohn}, {\em {Inhomogeneous electron gas}}, Phys. Rev.,
  136 (1964), pp.~B864--B871.

\bibitem{Hormander1961}
{\sc L.~H{\"o}rmander}, {\em Hypoelliptic differential operators}, Ann. Inst.
  Fourier, 11 (1961), pp.~477--492.

\bibitem{KohnSham1965}
{\sc W.~Kohn and L.~Sham}, {\em {Self-consistent equations including exchange
  and correlation effects}}, Phys. Rev., 140 (1965), pp.~A1133--A1138.

\bibitem{Laury2015}
{\sc M.~L. Laury, L.~P. Wang, V.~S. Pande, T.~Head-Gordon, and J.~W. Ponder},
  {\em Revised parameters for the amoeba polarizable atomic multipole water
  model}, J Phys Chem B, 119 (2015), pp.~9423--9437.

\bibitem{LeimkuhlerMatthews2015}
{\sc B.~Leimkuhler and C.~Matthews}, {\em Molecular Dynamics}, Springer-Verlag
  New York, 2015.

\bibitem{LeimkuhlerMatthewsStoltz2015}
{\sc B.~Leimkuhler, C.~Matthews, and G.~Stoltz}, {\em The computation of
  averages from equilibrium and nonequilibrium langevin molecular dynamics},
  Ima J. Numer. Anal., 36 (2015).

\bibitem{LinLuShao2014}
{\sc L.~Lin, J.~Lu, and S.~Shao}, {\em Analysis of the time reversible
  {Born-Oppenheimer} molecular dynamics}, Entropy (Special issue on Molecular
  Dynamics Simulation), 16 (2014), pp.~110--137.

\bibitem{Martin2004}
{\sc R.~Martin}, {\em Electronic Structure -- Basic Theory and Practical
  Methods}, Cambridge Univ. Pr., West Nyack, {NY}, 2004.

\bibitem{Niklasson2008}
{\sc A.~M.~N. Niklasson}, {\em {Extended Born-Oppenheimer molecular dynamics}},
  Phys. Rev. Lett., 100 (2008), p.~123004.

\bibitem{Niklasson2012}
{\sc Anders M.~N. Niklasson and Marc~J. Cawkwell}, {\em {Fast method for
  quantum mechanical molecular dynamics}}, Phys. Rev. B, 86 (2012), p.~174308.

\bibitem{Niklasson2009}
{\sc A.~M.~N. Niklasson, P.~Steneteg, A.~Odell, N.~Bock, M.~Challacombe, C.~J.
  Tymczak, E.~Holmstr\"{o}m, G.~Zheng, and V.~Weber}, {\em {Extended Lagrangian
  Born-Oppenheimer molecular dynamics with dissipation}}, J. Chem. Phys., 130
  (2009), p.~214109.

\bibitem{Niklasson2006}
{\sc A.~M.~N. Niklasson, C.~J. Tymczak, and M.~Challacombe}, {\em
  {Time-reversible Born-Oppenheimer molecular dynamics}}, Phys. Rev. Lett., 97
  (2006), p.~123001.

\bibitem{PardouxVerrtennikov2001}
{\sc E.~Pardoux and A.~Yu. Verrtennikov}, {\em On the poisson equation and
  diffusion approximation. i}, Ann. Probab., 29 (2001), pp.~1061--1085.

\bibitem{PardouxVerrtennikov2003}
\leavevmode\vrule height 2pt depth -1.6pt width 23pt, {\em On poisson equation
  and diffusion approximation 2}, Ann. Probab., 31 (2003), pp.~1166--1192.

\bibitem{PardouxVerrtennikov2005}
\leavevmode\vrule height 2pt depth -1.6pt width 23pt, {\em On poisson equation
  and diffusion approximation 3}, Ann. Probab., 33 (2005), pp.~1111--1133.

\bibitem{Pavliotis2014}
{\sc G.~A. Pavliotis}, {\em Stochastic Processes and Applications},
  Springer-Verlag New York, first~ed., 2014.

\bibitem{PavliotisStuart2008}
{\sc G.~A. Pavliotis and A.~M. Stuart}, {\em Multiscale Methods},
  Springer-Verlag New York, 2008.

\bibitem{Talay2002}
{\sc D~Talay}, {\em Stochastic hamiltonian systems: Exponential convergence to
  the invariant measure, and discretization by the implicit euler scheme},
  Markov Process Relat., 8 (2002), pp.~163--198.

\bibitem{ReaxFF}
{\sc A.~C~T Van~Duin, Siddharth Dasgupta, Francois Lorant, and William~A.
  Goddard}, {\em Reaxff: A reactive force field for hydrocarbons}, Journal of
  Physical Chemistry A, 105 (2001), pp.~9396--9409.

\bibitem{Villani2009}
{\sc C.~Villani}, {\em Hypocoercivity}, Mem. Amer. Math. Soc., 202 (2009).

\bibitem{Wang2005}
{\sc Wei Wang and Robert~D. Skeel}, {\em Fast evaluation of polarizable
  forces}, The Journal of Chemical Physics, 123 (2005), p.~164107.

\bibitem{Wilcox1967}
{\sc R.~M. Wilcox}, {\em Exponential operators and parameter differentiation in
  quantum physics}, J. Math. Phys., 8 (1967), p.~962.

\end{thebibliography}

\end{document}